\documentclass[12pt]{article}
\usepackage{amsmath,amsfonts,amssymb,amsthm,amscd}
\usepackage{stmaryrd}
\title{Generic Stability and Modes of Convergence
}

\author{Karim Khanaki\thanks{Partially supported by IPM grant 1401030117.}\\Arak University of Technology}

\newtheorem{Theorem}{Theorem}[section]
\newtheorem{Proposition}[Theorem]{Proposition}
\newtheorem{Definition}[Theorem]{Definition}
\newtheorem{Remark}[Theorem]{Remark}
\newtheorem{Lemma}[Theorem]{Lemma}
\newtheorem{Corollary}[Theorem]{Corollary}
\newtheorem{Fact}[Theorem]{Fact}
\newtheorem{Example}[Theorem]{Example}

\newtheorem{Observation}[Theorem]{Observation}
\newtheorem{Convention}[Theorem]{Convention}

\def\dotminus{\mathbin{\ooalign{\hss\raise1ex\hbox{.}\hss\cr
  \mathsurround=0pt$-$}}}

\begin{document}
\maketitle

\begin{abstract}
We expand the study of generic stability in three different directions. Generic stability is best understood as a property of types in $NIP$ theories in classical logic. In this paper, we make attempts to generalize our understanding to Keisler measures instead of types, arbitrary theories instead of $NIP$ theories, and continuous logic instead of classical logic.
For this purpose, we study randomization of first-order structures/theories and modes of convergence of types/measures.
\end{abstract}

\section{Introduction} \label{1} 
This paper is a kind of companion-piece to \cite{K-Morley} on analytical contents of generically stable types/measures for   arbitrary theories.
Pillay and Tanović \cite{Pillay-Tanovic} introduced the notion of {\em generically stable type}, for {\em arbitrary} theories. 
 The notion  abstracts/expresses/preserves the crucial properties of definable types in stable theories, and definable and finitely satisfiable   types in $NIP$ theories. 
In \cite{HPS}, this notion was {\em correctly} generalized to {\em Keisler measures}, for $NIP$ theories. Assuming $NIP$, there are also equivalences on the notion, namely $dfs$, $fam$, and $fim$.\footnote{These are abbreviations for definable and finitely satisfiable  measure, finitely approximated measure, and frequency interpretation measure, respectively. Cf. \cite{CGH} for definitions.}
The important question that remains is what is the ``correct" notion of  generic stability for measures in {\em arbitrary} theories. This seems to be more complicated than what is expected in the first encounter.

Recently \cite{K-Morley}, it is shown that a global type $p$ in a  {\em countable} theory is generically stable over a small model $M$ if and only if $p$ is definable over $M$, and  there is a sequence $(c_i:i<\omega)$ inside $M$ such that $(tp(c_i/{\cal U}):i<\omega)$ converges to $p$, {\bf in a strong sense.}\footnote{$DBSC$-convergence. Cf. \cite[Def.~3.4]{K-Morley} or Remark~\ref{DBSC remark} below.} The significance of this result is that this notion is characterized by only a specific sequence in {\em the} model/set, and the information of all  Morley sequences is encoded in that sequence so that we can call it {\em  the} Morley sequence. On the other hand, it gives an analytical characterization of the notion which links it to other fields including topology/functional analysis/descriptive set theory.  This  suggests that a similar result holds for {\em measures} in arbitrary (countable) theories.

On the other hand, a line of research is randomization of first-order structures/theories. The randomization of a  complete first-order theory was introduced by Keisler \cite{Keisler} and formalized as a metric theory by Ben Yaacov and Keisler \cite{Ben-Keisler}. Intuitively, for a (complete) classical theory $T$ the randomization $T^R$ of $T$ is a (complete) continuous theory whose elements of models  are random elements of $M$ for some model $M$ of $T$. A default theme is that model-theoretic properties are preserved in randomization.\footnote{Although, there are counterexamples that this is not always the case. For example, randomization of an unstable simple theory is not simple \cite{Ben2}.} As an example $T$ is stable/$NIP$ if and only if $T^R$ is so. It is known that ``measures" in classical logic correspond to ``types" in randomization. With these ideas in mind, and since the notion of ``generic stability for types" is {\em practically} known even for continuous logic, we are looking for some properties  of measures in classical logic which are  {\em close to} generic stability of types in randomization.

The present paper aims  to investigate    `generic  stability' for measures in {\em arbitrary} theories. (We focus more on   measures in classical logic which are not types, and types in continuous logic/randomization.) For this, we first generalize/adapt some results of \cite{K-Morley} on generically stable types to continuous logic, and then we transfer these results to measures in classical logic using randomization. 
We define the notions of $R^{\text type}$-generic/$R$-generic  stability, as a property of  types/measures in classical logic, and 
show that  a type is $R^{\text type}$-generically stable 
 if and only if it is generically stable in the usual sense.
  We will see that, in countable languages, for every measure,  $R$-generic stability and  $fam$ are equivalent. 
   We study generic stability of measures through two procedures for transferring Keisler measures to types in the randomization, namely ``natural extension" and ``corresponding random-type". 
   We also provide different proofs of some well-known results by assuming $NIP$.  
   Finally,  we prove that a type is $R^{\text type}$-generically stable  if and only if   its corresponding random-type is generically stable (over a model of the form $M\otimes\mathcal{A}$).\footnote{Cf. Subsection 3.2, for definition of models of the form $M\otimes{\cal A}$}

To simplify reading  through the paper, we list some more important results: \ref{Morley sequence}, \ref{generic type continuous}, \ref{Grothendieck}, \ref{continuous fim=gs},   \ref{generic-type}, \ref{fam=R-generic}, \ref{main}, \ref{main completion},  \ref{Ben local}, \ref{type-fim=fim}.

  Let us   give our motivation and background. First, {\em we believe that  model theory has a deep analytical nature that is not yet fully studied.}
  In \cite{KP}, \cite{K3}, \cite{K-Baire}, \cite{K-GC}, \cite{K-Dependent} and \cite{K-Morley}, some analytical aspects of model theory/classification theory   were studied. Roughly, it is shown that  some model-theoretic notions appeared independently in topology/functional analysis/descriptive set theory, and moreover
  various characterizations yield  the characterization
  of $NOP$/$NIP$/$NSOP$ in a model $M$ or set $A$, and some important theorems in model theory have twins there.
  Also,  there are connections between classification in model theory
  and classification of Baire class 1 functions which lead to a better understanding of both of these topics \cite{K-Baire}. 
  The key idea is that the study of the model-theoretic properties of formulas in `models' instead of only these properties in
  `theories' develops a sharper stability theory and establishes important links
  between model theory and other areas of mathematics.
  In the present paper, we continue this approach  and complete/generalize some results  of  \cite{K-Dependent} and \cite{K-Morley}. Another work that has influenced this article is \cite{Gannon sequential}, in  which an analytic aspect  of generically stable types in (countable) theories is given (cf. \cite[Thm~4.8]{Gannon sequential}). And of course, the note \cite{Ben-transfer} on finite satisfiabiliy in randomization  and the paper \cite{Ben-VC} on continuous $VC$-theory have influenced our work. 
  
It is worth mentioning that two preprints  \cite{CGH-generic} and \cite{Gannon transfer} have been released shortly after our paper which address some similar questions. The aims  of Sections 1 to 3 in \cite{CGH-generic} are similar to those in our paper, but the approaches and methods in our work and theirs are completely different. 
Sections 1 and 2 in \cite{Gannon transfer} clarify some of the definitions and claims of the note \cite{Ben-transfer}.
Although we have not utilized the results of these two preprints, for greater clarity regarding the results/claims we obtained from the note \cite{Ben-transfer} and to assist the reader of this text, we refer to some definitions and results in them.

\begin{Convention}  (1): Although many of our results can be generalized to uncountable case as well, in this paper, we study {\bf countable} theories, classic or continuous. When we say $T$ is a {\bf theory} we mean a
 countable {\bf classical} theory, and  we say $T$ is a  {\bf continuous theory} otherwise; of course it is still countable. Also, all theories are complete.
\newline 
(2): The monster model for a classical theory is denoted by $\cal U$ and for a continuous theory is denoted by $\Bbb U$. This distinct is important, especially when we point  out that any model of the randomization of the form ${\cal U}\otimes\mathcal{A}$ is a {\bf strict} substructure of $\Bbb U$.\footnote{The difference between models of the form $M\otimes\mathcal{A}$ and the other models are important in this paper.  (Cf. Remark~\ref{warning}.)}
\newline 
(3): In this paper, when we say that $(a_i)\subset {\cal U}$ (or $(a_i)\subset {\Bbb U}$) is a sequence, we mean the usual notion in the sense of analysis. That is, every sequence is indexed by $\omega$. Similarly, we consider Morley sequences  indexed by $\omega$.
\newline
(4): In this paper, a variable $x$ is a tuple of length $n$ (for $n<\omega$).\footnote{Although all arguments are true for infinite variables, to make the proofs readable, we consider finite tuples.} Sometimes we write $\bar x$ or $x_1,\ldots,x_n$ instead of $x$. All types are $n$-types  (for $n<\omega$) unless explicitly stated otherwise. Similarly, a sequence  $(a_i)\subset {\cal U}$ (or $(a_i)\subset {\Bbb U}$) is a sequence of tuples of  length $n$ (for $n<\omega$).
\end{Convention}
 
 
 \medskip
 This paper is organized as follows. In Section 2,   we generalize/adapt some results of \cite{K-Morley} on generically stable types to continuous logic.   We give characterizations of generic stability of continuous types in the terms of convergence of (Morley) sequences of types. We also show that generically stable continuous types are $fim$. 
In Section 3, we introduce the notions  $R^{\text type}$-generic/$R$-generic stability for types/measures in classical logic. We show that a type  (in classical logic) is $R^{\text type}$-generically  stable if and only if it is generically stable in  the usual sense. In measures case, we show that 
$R$-generic stability and $fam$ are equivalent. We also give some results on $fim$ and  $fam$ measures and their corresponding random-types.
 In Section~4, we   prove that a type is $R$-generically stable if and only if its corresponding random-type is    generically stable.

\section{Continuous logic and generic stability} \label{2} 
In this section we introduce the notion of generic stability for types in {\em continuous logic} and give  characterizations of this notion which will be used later. The proofs are  adaptations of the arguments in classical logic \cite{K-Morley}.\footnote{This section can't be read without firm grasp of \cite{K-Morley}.}

{\bf If the reader thinks that it is not necessary to state, repeat and review the results of classical case (i.e. \cite{K-Morley}) for continuous logic, we must give the following warnings:} (a) Not all results of classical logic can be directly translated into continuous logic. For example, Shelah's theorem (i.e. stable=$NIP$+$NSOP$) cannot be  translated into continuous logic (cf. \cite{K-classification}). (b) It is important to have {\em correct} definitions in continuous logic and to check that the proofs work  in this case, and why these proofs still work or do not. (c) The results in continuous case sometimes lead to new results even in classical logic, as   we will see in Section~3 of the present paper and as   was seen in \cite{K-classification}.

 Therefore, we introduce the notion  in detail  but  we will refer the proofs to the classical case  and explain why arguments in classical case work here as well. We assume familiarity with  the basic notions about continuous model theory as developed in \cite{BBHU}. 

In this section,  we fix a  {\bf continuous} first order language $L$,\footnote{As mentioned earlier, in this paper, the language is countable, however we can consider {\em separable} languages. Also, we can study   separable fragments of the language. To make the proofs more readable, we assume that the theory is countable.} a complete  countable ({\bf continuous}) $L$-theory $T$ (not necessarily $NIP$), and a small set $A$ of the  monster model  $\Bbb U$. The space of global types in the variable $x$ is denoted by $S_x(\Bbb U)$ or $S(\Bbb U)$.

In the following, $\phi(\bar x)$ is a formula, $r,s$ are real numbers in $[0,1]$, and $\phi(\bar x)\leq r$, $\phi(\bar x)\geq r$, and $\phi(\bar x)= r$ are $L$-statements  (or $L$-conditions) in continuous logic.

Recall that, in continuous logic, we use $\sup,\inf,\min,\max, \leq$ instead of $\forall,\exists, \wedge, \vee,\dotminus$, respectively.
However, we will sometimes continue to use classic symbols to make our article more readable.

\begin{Remark} \label{diagonal work}
	Notice that $L$ is  countable and we assume that $M$ is a separable model (i.e., there is a countable dense subset $M_0\subseteq M$). We work with a countable system of connectives containing $0,1,\min,\max,\cdot/2,\dotminus$ (where $x\dotminus y=\max(x-y,0)$).
	 Recall from \cite[Thm.~6.3]{BBHU} that this does not impose any restrictions. 
	In this paper, we can work with $L(M_0)$-statements of the form $\phi(\bar x)\leq r$ and $\phi(\bar x)\geq s$ where $r,s$ are rational numbers in $[0,1]$. Therefore, the set of all statements  is  countable and all diagonal arguments in \cite{K-Morley} work  well in the present paper.
\end{Remark}

\begin{Definition}
	Let   $A\subset\Bbb U$ and $\phi(x_1,\ldots,x_n)\in L(A)$. We say that $\phi(x_1,\ldots,x_n)$ is {\em symmetric} if for any permutation $\sigma$ of $\{1,\ldots,n\}$, $$\sup_{\bar x}\big|\phi(x_1,\ldots,x_n)-\phi(x_{\sigma{(1)}},\ldots,x_{\sigma{(n)}})\big|=0.$$
\end{Definition}

\begin{Example} \label{example} 
Let $\phi(x,y)$ be a formula, and $r<s$. The following formula $\theta_{\phi,n}^{r,s}(x_1,\ldots,x_n)$ is symmetric:
$$\forall F\subseteq\{1,\ldots,n\}\exists y_F\Big(\bigwedge_{i\in F}\phi(x_i,y_F)\leq r\wedge\bigwedge_{i\notin F}\phi(x_i,y_F)\geq s\Big).$$
This formula should play  a key role in the arguments in  the present paper. See Remark~\ref{converges} and Theorem~\ref{Morley sequence} below, and also   \cite[Thm.~2.8]{K-Morley}.
\end{Example}

\begin{Remark} \label{converges}
	Let $(a_i:i<\omega)$ be an indiscernible sequence of $x$-tuples and $\phi(x,y)$ a formula. Then, it is easy to check that the following are equivalent:
	\newline
	(i) For any parameter $b\in\Bbb U$, the truth value of sequence $(\phi(a_i,b):i<\omega)$ is eventually constant.
		\newline
	(ii) For any $r<s$ there is a natural number $n$ such that  $\models \neg \ \theta_{\phi,n}^{r,s}(a_1,\ldots,a_n)$.
\end{Remark}

The following is an adaptation/generalization of \cite[Def.~2.3]{K-Morley}.\footnote{The EEM-type in classical logic is defined in \cite[Def. 4.3]{Gannon sequential}. It was extracted from
	the notion of eventual indiscernible sequence in \cite{S-invariant}.}
\begin{Definition}
	(i) Let $(b_i)$ be a sequence of $\Bbb U$ and $A\subset\Bbb U$  a   set. The   eventual Ehrenfeucht-Mostovski type  (abbreviated $EEM$-type) of $(b_i)$ over $A$, which is denoted by $EEM((b_i)/A)$, is the following (partial) type in $S_\omega(A)$:
	$$\big(\phi(x_1,\ldots,x_k)=r\big)\in EEM((b_i)/A)  \iff \lim_{i_1<\cdots<i_k, \  i_1\to\infty} \phi(b_{i_1},\ldots,b_{i_k})=r.$$
	(ii) Let $(b_i)$ be a sequence of $\Bbb U$ and $A\subset\Bbb U$  a   set. The symmetric eventual Ehrenfeucht-Mostovski type (abbreviated $SEEM$-type) of $(b_i)$ over $A$, which is denoted by $SEEM((b_i)/A)$, is the following partial type in $S_\omega(A)$:
	$$\Big\{\phi(x_1,\ldots,x_k)=r:\big(\phi(\bar x)=r\big)\in EEM((b_i)/A) \text{ and }\phi \text{ is symmetric} \Big\}.$$ 
	Whenever $(b_i)$ is $A$-indiscernible, we sometimes write $SEM((b_i)/A)$ instead of  $SEEM((b_i)/A)$.
	\newline
	(iii)  Let  $p(x)$ be a type in $S_\omega(A)$ (or  $S_\omega({\cal U})$). The symmetric type of $p$, denoted by Sym$(p)$, is   the following partial type: $$\Big\{\big(\phi(x)=r\big)\in p: \phi \text{ is symmetric} \Big\}.$$
\end{Definition}
The  sequence $(b_i)$  is called {\em eventual indiscernible over $A$} if $EEM((b_i)/A)$ is a {\em complete} type. In this case, for any $L(A)$-formula $\phi(x)$, the limit $\lim_{i\to\infty}\phi(b_i)$ is well-defined.

\medskip
Let  $p(x)$ be a global $A$-invariant type. The Morley type (or sequence) of  $p(x)$ can be easily defined similar to classical logic.\footnote{Cf. \cite{Simon}, subsection 2.2.1 for definition in classical case.} The Morley type (or sequence) of  $p(x)$ is denoted by $p^{(\omega)}$.  The restriction of $p(x)$ to $A$ is denoted by $p|_{A}$.
A realisation $(d_i:i<\omega)$ of $p^{(\omega)}|_A$ is called a Morley sequence of/in $p$ over $A$.

\begin{Lemma} \label{Key lemma}
	Let $p(x)\in S(\cal \Bbb U)$ be finitely satisfiable in $M$ where $M$ is a  separable model.\footnote{We can assume that $M$ is a separable set in some model.} Then  there is a sequence $(c_i)$ in $M$ such that $SEEM((c_i)/M)=\text{Sym}(p^{(\omega)}|_M)$.
\end{Lemma}
\begin{proof} 
		Let $I=(d_i)$ be a Morley sequence in $p$ over $M$. Since $T$ is countable and $M$ is separable, there is a sequence $(c_i)$ in $M$ such that $\lim tp(c_i/MI)=p|_{MI}$.\footnote{Notice that, in this case, the space $S(MI)$ is metrizable.} We can assume that $(c_i)$ is eventual indiscernible over $MI$. That is, the type $EEM((c_i)/MI)$ is complete. (Notice that, as $MI$ is separable and $T$ is countable, using Ramsey's theorem and a diagonal argument, there is a subsequence of $(c_i)$ which is eventual indiscernible over $MI$. See also Remark~\ref{diagonal work} above.)
We claim that $(c_i)$ is the desirable sequence. That is, $SEEM((c_i)/M)=SEM((d_i)/M)=\text{Sym}(p^{(\omega)}|_M)$. The proof is by induction on symmetric formulas similar to  \cite[Lemma~2.7]{K-Morley}.
\end{proof}

We say that a  sequence $(d_i)\in{\Bbb U}$ of $x$-tuples  converges (or is convergent)  if the sequence $(tp(d_i/{\Bbb U}):i<\omega)$ converges in the logic topology. Equivalently, for any $L({\Bbb U})$-formula $\phi(x)$, the sequence  $(\phi(d_i):i<\omega)$ is convergent, i.e. for any $\epsilon>0$ there is an natural number $n$ such that $|\phi(d_i)-\phi(d_j)|<\epsilon$ for all $i,j\geq n$. If $(tp(d_i/{\Bbb U}):i<\omega)$ converges to a type $p$, then we write $\lim tp(d_i/{\Bbb U})=p$ or $tp(d_i/{\Bbb U})\to p$. Notice that $tp(d_i/{\Bbb U})\to p$   if and only if   for any $L({\Bbb U})$-formula $\phi(x)$, $$\big(\phi(x)=r\big)\in p \iff \lim\phi(d_i)=r.$$

\begin{Theorem} \label{Morley sequence}
	Let $T$ be a  separable continuous theory, $M$ a separable model, and $p(x)\in S(\Bbb U)$ a global type which is finitely satisfiable in $M$. Let $(d_i)$ be a Morley sequence of $p$ over $M$. If $(d_i)$ converges then there is a sequence $(c_i)$ in $M$ such that $tp(c_i/\Bbb U)\to p$.
\end{Theorem}
\begin{proof} The proof is an adaptation of the argument of Theorem~2.11  of \cite{K-Morley}. Indeed, by  the argument of Lemma~\ref{Key lemma}, we can assume that  there is a sequence $(c_i)$ in  $M$ such that $tp(c_i/M\cup(d_i))\to p|_{M\cup(d_i)}$ and $SEEM((c_i)/M)=\text{Sym}(p^{(\omega)}|_M)$. We show that  $tp(c_i/{\Bbb U})\to p$. Let $q$ be an  accumulation point  of $\{tp(c_i/{\Bbb U}):i\in\omega\}$. Then $q|_{M\cup(d_i)}=p|_{M\cup(d_i)}$. Therefore, by an easy induction, we can show that the Morley types (sequences)  $p^{(\omega)}|_M$ and $q^{(\omega)}|_M$ are the same. Now, as $(d_i)$ converges, it is easy to show that $p=q$. (If not, using an standard  argument, one can build a Morley sequence $(a_i)$ such that for a parameter $b\in\Bbb U$, a formula $\phi(x,y)$, and for $r<s$, we have $\phi(a_i,b)<r$ if $i$ is even, and $\phi(a_i,b)>s$   otherwise. This contradicts the convergence of $(d_i)$. Cf.  Claim~1 in the proof of Theorem~2.11 in \cite{K-Morley}.) Also, using Rosenthal's lemma (cf. \cite[Fact~2.10]{K-Morley})\footnote{Notice that Rosenthal's lemma works in real-valued case too.}, similar to the Claim~2 in \cite[Thm.~2.11]{K-Morley}, we can show that  the sequence $(tp(c_i/{\Bbb U}):i<\omega)$ converges. If not, there are a formula $\phi(x,y)$, $r<s$, and a Morley sequence $(a_i:i<\omega)$ such that for any $n$ the formula $\theta_{\phi,n}^{r,s}$ holds in this Morley sequence. (Cf. Example~\ref{example}.) This means that $(a_i:i<\omega)$ diverges, a contradiction. These prove the theorem.
\end{proof}

The following is the counterpart of  Definition~\ref{Baire-1/2-convergence} below  to continuous logic.
\begin{Definition} \label{Baire-1/2-convergence 2}
	{\em  Let  $M$ be a model and $(p_n(x):n<\omega)$ be a sequence of   types over $M$. 
		We say that $(p_n:n<\omega)$ {\em  Baire-$1/2$-converges} (or {\em is  Baire-$1/2$-convergent}) if  the independence property is semi-uniformly  blocked on $(p_n:n<\omega)$, that is, for any formula $\phi(x,y)$,  and for each $r<s$, there is a natural number $N=N_{r,s}^\phi$ and a set $E\subset\{1,\ldots,N\}$ such that  for  each $i_1<\cdots<i_{N}<\omega$, and  any parameter $b\in M$, the following does not hold
		$$\bigwedge_{j\in E}\big(\phi(x,b)\leq r\big)\in p_{i_j}~\wedge~\bigwedge_{j\in N\setminus E}\big(\phi(x,b)\geq s\big)\in p_{i_j}.$$ 
Notice that Baire-1/2 convergence implies that the sequence is convergent.
Indeed, for any formula $\phi(x,y)$ and parameter $b$ the sequence ($\phi(p_n, b) : n<\omega)$ is a Cauchy sequence of real numbers and so is convergent.

\noindent	
	In the above, we can assume that $M$ is the monster model.}
\end{Definition}

\begin{Remark} \label{remark1}
Notice that in Theorem~\ref{Morley sequence}, the sequence $(tp(c_i/\Bbb U):i<\omega)$ is Baire-1/2-convergent. In fact, as $SEEM((c_i)/M)=\text{Sym}(p^{(\omega)}|_M)$, any/some Morley sequence is convergent if and only if $(tp(c_i/\Bbb U):i<\omega)$ is Baire-1/2-convergent. In particular, if the Morley sequence of $p$ is totally indiscernible (e.g. $p$ is generically stable), $$EEM((c_i)/M)=SEEM((c_i)/M)=\text{Sym}(p^{(\omega)}|_M)=p^{(\omega)}|_M.$$
This means that Morley sequences of $p$ are controlled by the sequence $(c_i)\in M$, and vice versa. We will shortly see that this fact leads to a new and useful characterization of the following notion, namely  generic stability.
\end{Remark}

\begin{Remark} \label{DBSC remark}
	Recall that, in {\bf classical} logic, the types are $\{0,1\}$-valued. In this case, we say that the sequence $(p_n:n<\omega)$ of types  {\em is  $DBSC$-convergent} {\em (or  $DBSC$-converges)} if it is  Baire-$1/2$-convergent as in Definition~\ref{Baire-1/2-convergence 2}.	More
	precisely, the sequence $(p_n:n<\omega)$ $DBSC$-converges if for any formula
	$\phi(x,y)$ there is a natural number $N = N_\phi$ such that for any parameter $b$,
	we have	$\sum_{n=1}^\infty |\phi(p_{n+1},b)-\phi(p_n,b)|\leq N$. (Cf. \cite[Def. 3.4]{K-Morley}.)  When $(p_n:n<\omega)$  is  $DBSC$-convergent, we say the independence property is uniformly  blocked on $(p_n:n<\omega)$. (Notice that, in this case, for any $r<s$ and formula $\phi(x,y)$, the natural number $N$ (in Definition~\ref{Baire-1/2-convergence 2}) depend {\bf just} on $\phi$.) Recall from \cite{K-GC} that these notions of convergence are related to different subclasses of Baire-1 functions. In fact, the limit of a $DBSC$-convergent (resp. Baire-$1/2$-convergent) sequence is a $DBSC$ (resp. Baire-$1/2$) function,  and the class of $DBSC$ functions is a {\em proper} subclasse of Baire-1/2 functions.
\end{Remark}

The following is the natural/correct adaptation of generic stability from \cite{Pillay-Tanovic} to continuous logic (cf. Remark~\ref{remark2}(v) below).
\begin{Definition}[Generic stability] \label{generic-type-continuous}
 Let $T$ be a continuous theory and $A$ a small set of the monster model. 	A global type $p(x)$ is {\em generically stable} over   $A$ if $p$ is $A$-invariant, and {\bf every} Morley sequence $(a_i:i<\omega)$ has the following properties:

\noindent
(i) $(a_i:i<\omega)$ is totally indiscernible, and 

\noindent
(ii) $(a_i:i<\omega)$ has no order; that is, there is no  sequence $(b_i:i<\omega)$,  formula $\phi(x,y)$, and $r<s$ such that $\phi(a_i,b_j)\leq r$ if $i<j$ and $\phi(a_i,b_j)\geq s$  otherwise.
\end{Definition}

\begin{Remark} \label{remark2} 
	Let $p(x)$ be a global $A$-invariant type.
	The following are equivalent.
	\newline (i) $p$ is generically stable over $A$.
	\newline (ii)  For {\bf any} Morley sequence $(a_i:i<\omega)$ of $p$ over $A$, we have $\lim tp(a_i/\Bbb U)=p$. 
		\newline (iii) The condition (ii) holds and furthermore, the sequences $(tp(a_i/\Bbb U):i<\omega)$ are Baire-$1/2$-convergent.
		\newline (iv)  Any Morley sequence of $p$ is totally indiscernible    and convergent.
		\newline (v) For {\bf any} Morley sequence $(a_i:i<\alpha)$  (any $\alpha$, not only $\omega$) of $p$ over $A$, and any formula $\phi(x)$  (with parameters from $\Bbb U$) and $r<s$, at least one of  $\{i: \ \models\phi(a_i)\leq r\}$ 
		or   $\{i:\ \models\phi(a_i)\geq s\}$ is  finite.
\end{Remark}
\begin{proof}
(i) $\Rightarrow$ (ii):
	As any Morley sequence $(a_i)$ is totally indiscernible and has no order, it is easy to see that $(a_i)$ is convergent. If not, there are a formula $\phi(x,y)$, $r<s$ and parameter $b\in\Bbb U$ such that the sets $\{i: \phi(a_i,b)\leq r\}$ and $\{i:\phi(a_i,b)\geq s\}$ are infinite.
Now, by total indiscernibility, one can find a sequence $(b_j)$ such that $\phi(a_i,b_j)\leq  r$ if $i<j$ and $\phi(a_i,b_j)\geq s$  otherwise, a contradiction. Now, a  (continuous)  generalization of Lemma 4.2 in \cite{K-Morley} implies that $\lim tp(a_i/{\Bbb U})=p$.

(ii) $\Rightarrow$ (i): As any Morley sequence is indiscernible and convergent to $p$, it is easy to see that  some/any Morley sequence is totally indiscernible. 
(Indeed, suppose for a contradiction that some/any Morley sequence $(a_i:i<\omega)$ is not totally indiscernible.  
As  $(a_i:i<\omega)$  is indiscernible, similar to the argument of Theorem 12.37 of \cite{P}, can someone find a formula $\psi(x,y)$ (with parameters)\footnote{This formula is indicated by $g(x,y,\bar a)$ in the proof of \cite[Thm. 12.37]{P}.}
 and the numbers $r<s$, such that this formula has the order property with the sequence $(a_i)_{k\leq i<\omega}$ (for some $k$) and $r<s$.
Now, someone can find a Morley sequence $(c_i:i<\omega+\omega)$ and a parameter $b$ such that $\psi(c_i,b)\leq r$
 if $i<\omega$ and $\psi(c_i,b)\geq s$ otherwise. This means that $\lim_{i<\omega}tp(c_i/{\Bbb U})\neq\lim_{i>\omega}tp(c_i/{\Bbb U})$, a contradiction, because for {\bf any} Morley sequence the limit is $p$.)
 Suppose for a contradiction that some Morley sequence $(a_i)$ has order for formula $\phi(x,y)$ and $r<s$; that is, there is a sequence $(b_j)$ such that $\phi(a_i,b_j)\leq r$ if $i<j$, and $\phi(a_i,b_j)\geq s$ otherwise.
Then, as $(a_i)$ is totally indiscernible, one can find a parameter $b\in\Bbb U$ such that $\phi(a_i,b)\leq r$ if $i$ is even, and $\phi(a_i,b)\geq s$ otherwise, a contradiction.
	
	(iii) follows from (ii) and the indiscernibility of Morley sequences. Indeed, if some Morley sequence is not  Baire-$1/2$-convergent, then there is a Morley sequence which is  not  convergent, a contradiction. (Cf. the direction  (i)~$\Rightarrow$~(ii) of Theorem~4.4 in \cite{K-Morley}, for classical case.)

(iv) $\Rightarrow$ (i) is given in  the proof of direction (ii) to (i).
	
	(i)~$\Rightarrow$~(iv) is given in the proof of direction (i) to (ii).

	
(iv) $\Rightarrow$ (v) is evident. If not, one can find a formula with order, and by (i)~$\iff$~(iv), a contradiction.

	(v)~$\Rightarrow$~(ii): Clearly  every Morley sequence $(a_i:i<\omega)$ is convergent. We need to show that its limit is $p$. If not, there are a formula $\phi(x,y)$, $r<s$ and parameter $b$ such that $\lim\phi(a_i,b)=r$ but $(\phi(x,b)=s)\in p$.
  Now, it is easy to find a Morley sequence $(c_i:i<\omega+\omega)$ such that  both sets $\{i:\phi(c_i,b)\leq r\}$ and $\{i:\phi(c_i,b)\geq s\}$ are infinite, a contradiction.
\end{proof}

The following theorem gives a characterizations of generically stable types for countable theories. The important one to note immediately is (ii).
\begin{Theorem}  \label{generic type continuous}
	Let $T$ be a  continuous theory, $M$ a small model of $T$, and $p(x)\in S(\Bbb U)$ a global $M$-invariant type. The following are equivalent:
	\newline
	(i) $p$ is generically stable over $M$.
	\newline
	(ii) $p$ is definable over a small model, and  there is a sequence $(c_i)$ in $M$ such that $(tp(c_i/\Bbb U):i<\omega)$ Baire-$1/2$-converges to $p$.
	\newline
	(iii) $p$ is definable over  and finitely satisfiable in some small model, and   there is a convergent Morley sequence of $p$ over $M$.
\end{Theorem}
\begin{proof}  The proof is an adaptation of the argument of Theorem~4.4 of \cite{K-Morley}. First, notice that, similar to classical case, we can assume that $M$ is separable.
	
		(i)~$\Longrightarrow$~(ii) follows from Theorem~\ref{Morley sequence}.
	(See also Remark~\ref{remark1}.)
	
		(ii)~$\Longrightarrow$~(i): As $p$ is definable and finitely satisfiable,  any Morley sequence is totally indiscernible. Let $(d_i)$ be a Morley sequence. Similar to Lemma~\ref{Key lemma}, 
		it is easy to see that $SEEM((c_i)/M)=Sym(tp((d_i)/M))$. Therefore, as $(c_i)$ is Baire-1/2-convergent, the Morley sequence $(d_i)$ converges. By Remark~\ref{remark2}(iv), $p$ is generically stable.

		(i)~$\Longrightarrow$~(iii) follows from Remark~\ref{remark2}(iv) and the fact that the Morley sequences of definable and finitely satisfiable types are totally indiscernible. (i)~$\Longrightarrow$~(iii) is evident.
\end{proof}

The following is a consequence of the previous results, although it has not  been stated anywhere before, even for classical logic.
\begin{Corollary} \label{Grothendieck}
Let $T$ be a  continuous theory, $M$ a small model of $T$, and $p(x)\in S(\Bbb U)$ a global $M$-invariant type. The following are equivalent:
\newline
(i) $p$ is generically stable over $M$.
\newline
(ii) There is a sequence $(c_i)$ in $M$ such that $(tp(c_i/\Bbb U):i<\omega)$ Baire-$1/2$-converges to $p$, and $(c_i)$ has no order, that is, there is no $(b_i)\in\Bbb U$, $r<s$, and formula $\phi(x,y)$ such that 
$\phi(c_i,b_j)\leq r$ if $i<j$ and $\phi(c_i,b_j)\geq s$   otherwise.
\end{Corollary}
\begin{proof}
	This follows from Theorem~\ref{generic type continuous} and   Grothendieck's double limit theorem (cf. \cite[Fact~2.2]{K-Baire}). Indeed, by Theorem \ref{generic type continuous}, there is a sequence $(c_i)$ in $M$ such that $(tp(c_i/{\Bbb U}):i<\omega)$ Baire-1/2 converges to $p$. We need to show that $(c_i)$ has no order. For any $c_i$ and  any formula $\phi(x,y)$, define map $F^\phi_{c_i}:S_{\phi^{opp}}(M)\to[0,1]$ via $q\mapsto \phi(c_i,b)$ for some/any $b\models q$, where $\phi^{opp}(y,x):=\phi(x,y)$. The set $A=\{F_{c_i}^\phi:i<\omega\}$ is a set of continuous functions on the type space $S_{\phi^{opp}}(M)$. Notice that, as $tp(c_i/{\Bbb U})\to p$, the  pointwise closure of $A$ is the set $A\cup\{F_p^\phi\}$,  where $F_p^{\phi}(q)=r$  if  $(\phi(x,b)=r)\in p$ for some/any $b\models q$.\footnote{As an easy  analytical exercise, we leave it to the reader to examine that the pointwise closure of $A$ is $A\cup\{F_p^{\phi}\}$.}
 By Grothendieck's theorem, $F_p^{\phi}$ is continuous  if and only if there is no sequence $(b_j)$ and $r<s$ such that $\phi(c_i,b_j)\leq r$ if $i<j$ and $\phi(c_i,b_j)\geq s$ otherwise. Recall that $p$ is definable if and only if for any formula   $\phi(x,y)$ the function $F_p^\phi$ is continuous. This proves the claim.
\end{proof}

\subsection*{Continuous types and $fim$} \label{4} 
In this section we introduce the notion $fim$ to continuous types and show that any generically stable type in continuous logic is $fim$.

Let $\phi(x,y)$ be a formula, and ${\bar a}=(a_1,\ldots,a_n)$ and $b$ parameters. Then we define $Av({\bar a})(\phi(x,b)):=\frac{1}{n}\sum_{i=1}^n\phi(a_i,b)$. 
\begin{Definition}[Continuous $fim$] \label{continuous fim}
	{\em  Let $p(x)$ be a continuous type and $A$ a small set. We say that $p$ is $fim$ if it is $A$-invariant  and for any formula $\phi(x,y)$ and $\epsilon>0$, there is an  $L(A)$-statement $\theta(x_1,\ldots,x_n)$ such that:
		
		(i) $\theta(x_1,\ldots,x_n)\in p^{(n)}$, and 
		
		(ii) for all $\bar a\models \theta(x_1,\ldots,x_n)$ we have $$\sup_{b\in{\Bbb U}}|p((\phi(x,b))-Av({\bar a})(\phi(x,b))|\leq\epsilon.$$ }
\end{Definition}

\begin{Lemma} \label{lemma}
	Let $p(x)$ be a continuous type. Suppose that its Morley type/sequence $p^{(\omega)}$ is totally indiscernible and  convergent. Then for any formula $\phi(x,y)$ and $\epsilon>0$ there is a natural number $n_{\phi,\epsilon}$ such that for any Morley sequence $(a_i:i<\omega)\models p^{(\omega)}$ we have:
	\newline
	for any $b\in\Bbb U$, the number of $i$ such that $|\phi(a_i,b)-\lim_i\phi(a_i,b)|>\epsilon$ is $\leq n_{\phi,\epsilon}$.
\end{Lemma}
\begin{proof}
	If not, using total indiscernibility and compactness, one can find a divergent Morley sequence, a contradiction.
\end{proof}

\begin{Proposition} \label{continuous fim=gs}
	Let $p(x)$ be a continuous type. If $p$ is generically stable, then it is $fim$.
\end{Proposition}
\begin{proof}
	The proof is an adaptation of Proposition~3.2 of \cite{CG}. 
	As $p$ is definable, for any formula $\phi(x,y)$ and any $\epsilon>0$, there is  a formula $\psi_{\phi,\epsilon}=\psi(y)$ (with parameters), which is a finite continuous combination of the instances of  $\phi(a,y)$, such that 
	$\sup_{b\in{\Bbb U}}|p(\phi(x,b))-\psi(b)|\leq\epsilon$. 
	Recall that, as $p$ is generically stable, any Morley sequence $(a_n:n<\omega)$ is totally indiscernible, and $\lim tp(a_n/\Bbb U)=p$.
	
	Define the statement $\theta(x_1,\ldots,x_n)$ as follows:
	$$\forall y \bigwedge_{I\subset n, |I|>n_{\phi,\epsilon} }\Big(\exists J\subset I, |J|\geq |I|- n_{\phi,\epsilon} \ \ \big( \bigwedge_{i\in J}\big|\phi(x_i,y)-\psi(y)\big|\leq2\epsilon\big)\Big).$$
	
	We will show  that for big enough $n\gg n_{\phi,\epsilon}$, $\theta(x_1,\ldots,x_n)\in p^{(n)}$, and  the condition (ii) of Definition~\ref{continuous fim} holds.	 Indeed, for any Morley sequence $(a_i)$,  as $\lim tp(a_i/{\Bbb U})=p$, we have $|\lim\phi(a_i,b)-\psi(b)|\leq\epsilon$ for any $b\in\Bbb U$.
Notice that, for all $b$,  we have $|\phi(a_i,b)-\psi(b)|\leq|\phi(a_i,b)-\lim\phi(a_i,b)|+|\lim\phi(a_i,b)-\psi(b)|\leq|\phi(a_i,b)-\lim\phi(a_i,b)|+\epsilon$.
 Therefore, by Lemma~\ref{lemma}, 
 the number $i$ such that $|\phi(a_i,b)-\psi(b)|>2\epsilon$ is $\leq n_{\phi,\epsilon}$.
 This means that 
 $\models\theta(a_1,\ldots,a_n)$, and so $\theta(x_1,\ldots,x_n)\in p^{(n)}$. On the other hand, for any ${\bar c}=(c_1,\ldots,c_n)$ where $\models\theta(\bar c)$, we have $|Av(\bar c)(\phi(x,b))-p(\phi(x,b))|\leq|Av(\bar c)(\phi(x,b))-\psi(b)|+|p(\phi(x,b))-\psi(b)|\leq|Av(\bar c)(\phi(x,b))-\psi(b)|+\epsilon$ for all $b\in \Bbb U$. As $\models\theta(\bar c)$, it is easy to see that $\sup_{b\in\Bbb U}|Av(\bar c)(\phi(x,b)-\psi(b)|\to 2\epsilon$ as $n\to\infty$. Therefore, for big enough $n$, we have $\sup_{b\in\Bbb U}|Av(\bar c)(\phi(x,b))-p(\phi(x,b))|\leq4\epsilon$. As $\epsilon$ is arbitrary, the proof is completed.
\end{proof}

\section{Measures and random types} \label{3} 
In this section we study generic stability of measures through two procedures for transferring Keisler measures to types in a continuous structure, namely the randomization. These two procedures are called {\em natural extension} (cf. Subsection~3.2) and {\em corresponding random-type} (cf. Subsection~3.3), although they sometimes go by other names in other literature. Wherever necessary, we refer to the literature for more clarity.

 In Subsection 3.1, we first introduce and study the notion of {\em $R$-generic stability}, as a property of a measure in {\bf classical} logic, and then in Subsections 3.2 and 3.3,  using the results of Section~\ref{2},  we study  this notion and related random-types in the randomization (i.e. natural extension and corresponding random-type) and their connections. 

\subsection{$R$-generic stability} \label{} 
In this subsection we introduce the notions of {\em $R^{\text{type}}$/$R$-generically  stable type/measure} in {\bf classical} logic, and show that for constable theories: (i) $R^{\text{type}}$-generic stability and the usual notion of generic stability for types are equivalent, and (ii) $R$-generic stability and $fam$ are equivalent for measures.


In this section,  we fix a  (classical) first-order language $L$, a complete  countable  $L$-theory $T$ (not necessarily $NIP$), and a small set $A$ of the  monster model  $\cal U$. The space of global types in the variable $x$ is denoted by $S_x({\cal U})$ or $S({\cal U})$.

\medskip
We first give a notion/notation. Let $\mu(x)$ be a global measure, $r_1,\ldots,r_k\in[0,1]$ such that $\sum r_i=1$. The measure $\mu^{\sum r_i}$ is defined as follows: for any formula $\phi(x,y)$, and any  parameters $b_1,\ldots,b_k$, $$\mu^{\sum r_i}(\phi;(b_i)_1^k):=\sum_1^kr_i\cdotp\mu(\phi(x,b_i)).$$

The following notion is technical but we will shortly see that it is the natural generalization of the corresponding notion in the case of types.
\begin{Definition} \label{Baire-1/2-convergence}
	{\em   Let  $(\mu_n(x):n<\omega)$ be a sequence of global $A$-invariant measures, and $\mu(x)$ a global $A$-invariant measure. 
	
	\noindent
	(i): We say that $(\mu_n:n<\omega)$ {\em is randomly convergent} (or {\em    randomly converges}) if for any formula $\phi(x,y)$ and for each $r<s$, there are a natural number $N=N_{r,s}^\phi$ and a set $E\subset\{1,\ldots,N\}$ such that for any  $r_1,\ldots,r_k\in[0,1]$ with $\sum r_t=1$  and any  parameters $b_1,\ldots,b_k$ and  for  each $i_1<\cdots<i_{N}<\omega$, the following does not hold
		$$\bigwedge_{j\in E}\mu_{i_j}^{\sum r_t}(\phi;(b_t)_1^k)\leq r~\wedge~\bigwedge_{j\in N\setminus E}\mu_{i_j}^{\sum r_t}(\phi;(b_t)_1^k)\geq s. \ \ \  (*)$$
	
\noindent (ii): We say that $(\mu_n:n<\omega)$ randomly converges to $\mu$ if $(\mu_n:n<\omega)$
randomly converges and, additionally, $(\mu_n:n<\omega)$ converges to $\mu$ in the
space of measures.  }
\end{Definition}

\begin{Remark} \label{Baire 1/2 classic}
 In the above definition, if $(*)$ holds just for $k=1$, then we say that $(\mu_n:n<\omega)$ {\em is Baire-1/2-convergent} (or {\em     Baire-1/2 converges}). (Campare   to Definition~\ref{Baire-1/2-convergence 2} above.) A question arises:
	is every   Baire-1/2-convergent sequence randomly convergent?  
\end{Remark}

\begin{Definition} \label{generic-measure}
	(i) Let  $\mu(x)$ be a global  measure. 
	We say that $\mu$ is {\em $R$-generically stable} over   $A$ if $\mu$ is definable over $A$,  and  there is a sequence $(\mu_n:n<\omega)$ of  measures such that: 
	\newline 
	$\bullet$ $\mu_n=\frac{1}{k_n}\sum_{i=1}^{k_n}p_{n,i}$ where $a_{n,i}\models p_{n,i}$ and $a_{n,i}\in A$ (for all $n$), and
	\newline
	$\bullet$ $(\mu_n:n<\omega)$ randomly converges to $\mu$.
	\newline
	(ii) Let  $p(x)$ be a global type. 
	We say that $p$ is {\em  $R^{type}$-generically stable} over   $A$ if $p$ is definable over $A$, and  there is a sequence $(p_n:n<\omega)$ of types such that: 
	\newline 
	$\bullet$   $a_n\models p_n$ and $a_n\in A$ (for all $n$), and
	\newline
	$\bullet$ $(p_n:n<\omega)$ randomly converges to $p$.
\end{Definition}
\begin{Remark}
(1): Notice that for a type $p$, the conditions in (ii) implies the conditions in (i),   but  there is no reason for the converse to be true.  
\newline(2): On the other hand, we can make a more subtle distinction. Indeed, for any finite set $F$ of real numbers in $[0,1]$, we can give a generalization of Definition~\ref{generic-measure}(ii) as follows. Let  $\mu(x)$ be a global $F$-valued measure.  We say that $\mu$ is {\em $R^F$-generically stable} over   $A$ if $\mu$ is definable over $A$, and  there is a sequence $(\mu_n:n<\omega)$ of  $F$-valued measures such that: 
\newline 
$\bullet$   $\mu_n$ is a  $F$-valued measure of a convex combination of types realized in   $A$ (for all $n$), and
\newline
$\bullet$ $(\mu_n:n<\omega)$ randomly converges to $\mu$.
\newline
Although we will not study it in this article, we believe that for a fixed $F$, the latter notion and the usual notion of generically stable type are of the same nature and complexity.
\end{Remark}

First we show  that, for types, $R^{type}$-generic  stability  and the usual notion of generic stability coincide. The following is a new characterization of generic stability which is interesting in itself.
\begin{Fact} \label{generic-type}
	Let  $p(x)$ be a global $M$-invariant type.
	The following are equivalent:
	
	(i) $p$ is generically stable over $M$ (as in \cite{Pillay-Tanovic}).

	(ii) $p$ is $R^{type}$-generically stable over $M$ (as in Definition~\ref{generic-measure}).
\end{Fact}
\begin{proof}
	We use a result from \cite{K-Morley}, namely Theorem~4.4. Using this theorem, $p$ is generically stable (in the usual sense)   if and only if   it is definable and  there is a sequence $(a_n)\in M$ such that $(tp(a_n/{\cal U}):n<\omega)$ $DBSC$-converges to $p$.\footnote{Cf. Remark~\ref{DBSC remark}, for definition of $DBSC$-convergent. Notice that for types Baire-1/2 convergent (as in Remark~\ref{Baire 1/2 classic}) and $DBSC$-convergent coincide.} It is easy to check that if  $(tp(a_n/{\cal U}):n<\omega)$ randomly converges, then it $DBSC$-converges. Conversely, if 
	$(tp(a_n/{\cal U}):n<\omega)$ $DBSC$-converges, then by Lemma~2.8 in \cite{K-Baire}, for any formula $\phi(x,y)$ there is a natural number $m$ such that for any parameters $b$,
	$$\sum_1^\infty|\phi(a_n,b)-\phi(a_{n+1},b)|\leq m.$$
	Now, by the triangle inequality, it is easy to check that for any  $r_1,\ldots,r_k\in[0,1]$ with $\sum r_i=1$ and parameters $(b_i)_1^k$, we have
	$$\sum_{n=1}^\infty\Big|\sum_{i=1}^k r_i\cdotp\phi(a_n,b_i)-\sum_{k=1}^k r_i\cdotp\phi(a_{n+1},b_i)\Big|\leq m. \ \ \ (\dagger)$$
(Notice that $(\dagger)$ is Fubini.) 
\newline
This implies that  $(tp(a_n/{\cal U}):n<\omega)$ randomly converges.
Indeed, suppose for a contradiction that there are parameters $(b_i)_1^k$, natural number $N>\frac{m}{|r-s|}$, $r_1,\ldots,r_k\in[0,1]$ with $\sum_1^k r_i=1$, and $j_1<\ldots<j_N<\omega$ such that $\mu_{j_t}^{\sum r_i}(\phi;(b_i)_1^k)\leq r$ if $t$ is odd, and $\mu_{j_t}^{\sum r_i}(\phi;(b_i)_1^k)\geq s$ if $t$ is even. (Here $\mu_{n}^{\sum r_i}(\phi;(b_i)_1^k)=\sum_{i=1}^k r_i\cdotp\phi(a_n,b_i)$.) Then
\begin{align*}
\sum_{n=1}^\infty\big|\mu_{n}^{\sum r_i}(\phi;(b_i)_1^k)-\mu_{n+1}^{\sum r_i}(\phi;(b_i)_1^k)\big| &  \geq\sum_{t=1}^N\big|\mu_{j_t}^{\sum r_i}(\phi;(b_i)_1^k)-\mu_{j_{t+1}}^{\sum r_i}(\phi;(b_i)_1^k)\big| \\
&  \geq  N\cdot |r-s| \\
& >m.
\end{align*} 
This contradicts $(\dagger)$.
\end{proof}

\subsection*{Representation of $fam$} 
We want to show that $R$-generic stability is a representation of $fam$. This  is useful in the next section.

First, we need to recall from \cite[Definition~3.1]{Gannon sequential} the notion of sequential approximation of measures. Let  $\mu(x)$ be a global  measure. We say that $\mu$ is {\em sequentially approximated  over   $A$} if  there is a sequence $(\mu_n:n<\omega)$ of  measures such that: 
\newline 
$\bullet$ $\mu_n=\frac{1}{k_n}\sum_{i=1}^{k_n}p_i$ where $a_i\models p_i$ and $a_i\in A$ (for all $n$), and
\newline
$\bullet$ $(\mu_n:n<\omega)$   converges to $\mu$. 
\newline
Notice that convergence is in the logic topology or equivalently the topology of pointwise convergence. (Cf. \cite{Gannon sequential}, Definition~2.1.) Gannon \cite[Pro. 3.4]{Gannon sequential} showed that, in countable theories,  every definable and sequentially approximated measure is $fam$.

\medskip
The following observation was suggested to us by the referee.
\begin{Proposition} \label{fam=>R-generic}  Let  $\mu(x)$ be a global   measure. If $\mu$ is  $fam$ over   $A$, then 
	$\mu$ is  $R$-generically stable over $A$.
\end{Proposition}
\begin{proof}
	If $\mu$ is $fam$ over $A$, then $\mu$ is definable over $A$. It suffices to show that the random convergence   condition. As $T$ is countable, enumerate formulas from $L$, $(\phi_i)_{i<\omega}$.  Choose a sequence $(\mu_i:i<\omega)$ such that $\mu_n$ is a measure concentrating on finitely many realized types (in $A$) and for each $i\leq n$ and $b\in {\cal U}^y$,
	$$\big|\mu(\phi_i(x,b))-\mu_n(\phi_i(x,b))\big|<\frac{1}{n}.\footnote{In fact, as $\mu$ and the $\mu_n$'s are definable over $A$, one can easily verify that   for each $i\leq n$ and  any $A$-finitely satisfied measure $\nu$, we have
		$\big|\mu\otimes\nu(\phi_i(x,y))-\mu_n\otimes\nu(\phi_i(x,y))\big|<\frac{1}{n}$. However, the argement presented  here proves a lesser claim and is preliminary, although it suffies for the purposes of this article.}$$
	(This is possible using a construction in \cite[Pro.~3.3(iii)]{Gannon sequential}.)
	\newline
	We show that the sequence $(\mu_n:i<\omega)$ randomly converges to $\mu$. For a contradiction, suppose that there is a formula $\phi(x,y)=\phi_m(x,y)$ and numbers $r<s$ such that for any $N=\{0,\ldots,N\}$ and $E\subseteq N$, there exist $\bar r=r_1,\ldots,r_n, \bar b=b_1,\ldots,b_n$ and $i_1<\cdots<i_n$ such that 
	$$\bigwedge_{j\in E}\mu_{i_j}^{\sum r_t}(\phi;\bar b)\leq r~\wedge~\bigwedge_{j\in N\setminus E}\mu_{i_j}^{\sum r_t}(\phi;\bar b)\geq s.$$
	Let $n_1>\frac{1}{(r-s)/3}$ and $n_2=m$ (the index of our formula). Let $N=\{1,\ldots,\max\{n_1,n_2\}+2\}$ and $E=\{\max\{n_1,n_2\}+2\}$ and $k=\max\{n_1,n_2\}+1$. Then notice that for $\bar r$ and $\bar b$ and $i_1<\cdots<i_N$, if $\nu=\sum_{t=1}^n r_t\delta_{b_t}$ then 
	\begin{align*}
\mu_{i_k}^{\sum r_t}(\phi,\bar b) &  =\mu_k\otimes\nu(\phi(x, y)) \\
	& \approx_{\epsilon/3} \mu\otimes\nu(\phi(x,y)) \\
	& \approx_{\epsilon/3} \mu_{k+1}\otimes\nu(\phi(x,y)) \\
		& = \mu_{i_{k+1}}^{\sum r_t}(\phi,\bar b).
	\end{align*}
	(Here $c\approx_{\epsilon/3}d$   if and only if   $|c-d|\leq\epsilon/3$.)
	\newline
	Hence, if $\mu_{i_{k+1}}^{\sum r_t}(\phi,\bar b)\leq r$, then  $\mu_{i_k}^{\sum r_t}(\phi,\bar b)\leq r+2\epsilon/3$ and so it is not greater than or equal to $s$. To summarize, $(\mu_n:i<\omega)$ randomly converges to $\mu$.
\end{proof}

\begin{Corollary} \label{fam=R-generic}  Let  $\mu(x)$ be a global   measure.  $\mu$ is   $R$-generically stable over   $A$ if and only if
 $\mu$ is  $fam$ over $A$.
\end{Corollary}
\begin{proof} By Proposition~\ref{fam=>R-generic} above, we just need to show left to right. This follows from Proposition 3.4 of \cite{Gannon sequential}. First notice that, as $T$ is countable and $\mu$ is definable,   we can assume that $A$ is countable. Clearly, by definition, $\mu$ is sequentially approximated  over   $A$. As $\mu$ is definable and  sequentially approximated, Proposition 3.4 of \cite{Gannon sequential} implies that  $\mu$ is  $fam$.
\end{proof}
However, $fam$ and $R$-generic stability are the same, the latter presentation is useful in the following.  In the next section we study connections between   $R$-generic stability(=$fam$) and random-types in randomization.

\medskip
Recall that the notion of generically stable measures for $NIP$ theories was introduced and studied  in \cite{HPS}. The  next observation shows  that the new and usual notions coincide,  in $NIP$ theories.
\begin{Corollary} \label{}
	(Assuming $T$ is $NIP$.) Let  $\mu$ be a global measure. Then $\mu$ is   $R$-generically stable over   $A$ (as in Definition~\ref{generic-measure}(i))   if and only if   it is  generically stable over   $A$ (as in \cite{HPS}).
\end{Corollary}
\begin{proof}
	Recall that, in $NIP$ theories, any measure is  $fam$ if and only if it is  generically stable. Therefore, this follows from Corollary~\ref{fam=R-generic}. (In particular, for types, $R$-generic stability and $R^{type}$-generic stability   are the same.)
\end{proof}

\begin{Remark} \label{fam not generic}   In \cite{K-Morley}, we make a   claim: generic stability (for  types) is strictly stronger than \textbf{sad}.  (Recall from \cite{Gannon sequential} that a type is called \textbf{sad} if it is both \textbf{s}equentially \textbf{a}pproximated and \textbf{d}efinable.\footnote{We believe that \textbf{sad} types are sad because they are not genericaly stable.} Notice that sequential approximation for types and measures are differently defined in Gannon's paper.) Although, in \cite{K-Morley} we suggested an example of a non-generic and \textbf{sad}  type,  we have not found clear examples yet. 
\end{Remark}

\subsection{Random types and $R$-generic stability}
A  randomization of a first-order structure $M$, as formalized by Ben~Yaaocv and Keisler \cite{Ben-Keisler}, is a new {\em continuous} structure  whose elements are random elements  of $M$. In this section we show that a measure $\mu$ (in classical logic) is  $R$-generically stable  (as in Definition~\ref{generic-measure}) if   a canonical random-type (i.e. the natural extension) {\bf behaves  like}  generically stable types in   continuous logic (as in Theorem~\ref{generic type continuous}). 

We assume familiarity with  the basic notions about randomization of classical structures/theories as developed in \cite{Ben-Keisler} and \cite{Ben-transfer}.
Although we will recall  some notions and results from the research note \cite{Ben-transfer} in the following, and since this note is informal, we also refer to \cite[Subsection 3.2]{CGH-generic} and \cite[Sections 1 and 2]{Gannon transfer} for greater clarity and precision whenever necessary.

 In the following $T$ is a classical theory in the language $L$ and $T^R$ its randomization in the language $L^R$, as a continuous theory.

\begin{Convention} In this section, the symbol $\otimes$  is not used for the Morley product of types/measures, but will be used in another sense.
\end{Convention}

\begin{Convention}
	In the rest of the article, whenever necessary, we write the parameters in {\bf continuous} logic (i.e. in ${\Bbb U}$) in bold letters ${\bf a}, {\bf b}, {\bf c}\ldots$. Otherwise, we use $a, b, c,\ldots$.
\end{Convention}

Let $M$ be a classical $L$-structure and $\cal A$ an atomless measure algebra. The $L^R$-pre-structure $(M\otimes{\cal A})_0$ is defined as follows. The domain consists of all formal
finite sums $\sum_{i<k} m_i\otimes e_i$, also written $\bar m \bar e$, where $m\in M$ and $\bar e
=(e_i)_{i<k}\subseteq {\cal A}$ is a partition of the identity. If $e'$ is any other event then one can easily refine the partition and we identify members of 
$(M\otimes{\cal A})_0$ with other members obtained by refinement of partition. In this case, it is easy to check that:
$$f(\bar a\otimes\bar e,\bar b\otimes\bar e,\ldots )=(f(a_i,b_i,\ldots))\otimes \bar e,$$
$$\big[P(\bar a\otimes\bar e,\bar b\otimes\bar e,\ldots)\big]=\bigvee\{e_i:P(a_i,b_i,\ldots)\}\in\cal A.$$
As the distance symbol interprets a metric on $(M\otimes{\cal A})_0$,
its completion is denoted by $M\otimes{\cal A}$. Notice that if $M\models T$ then $M\otimes{\cal A}\models T^R$. (In \cite{CGH-generic}, the construction $(M\otimes{\cal A})_0$ is  denoted by $M_0^{\Omega}$. Cf. Subsection 3.2 in there.)

\begin{Remark}
Notice that every element ${\bf a}=\sum_{i<k}m_i\otimes e_i$  of $(M\otimes{\cal A})_0$ is a map from $\cal A$ to $M$ as follows: for each $i<k$, ${\bf a}(t)=m_i$ if $t\in e_i$. (Cf. \cite{CGH-generic} or \cite{Gannon transfer}.) This means that $\bf a$ is a simple function, i.e. its range is finite. Therefore, we can assume every element  $m$ of $M$ belongs to $(M\otimes{\cal A})_0$ via  the map $m\mapsto m\otimes 1$ where $m\otimes 1$ is the constant map $t\mapsto m$.
\end{Remark}
\begin{Convention} \label{simple writing}
In the rest of the paper, sometimes we write $m$ instead of $m\otimes 1$ for simplicity. Therefore, for an $L$-formula $\varphi(x)$ and parameter $m\in M$, we can write ${\Bbb P}[\varphi(m)]$ instead of ${\Bbb P}[\varphi(m\otimes 1)]$.
\end{Convention}

Let $\mu(\bar x)$ be a measure over $M$. Then, it is a random type over $M$, but not over  
$M\otimes{\cal A}$. Although there is a natural extension  $\mu\otimes{\cal A}$ of $\mu$ over $M\otimes{\cal A}$. For every $L$-formula $\phi$ we define 
$${\Bbb P}\Big[\phi\Big(\bar x,\sum m_ie_i\Big)\Big]^{\mu\otimes{\cal A}}=\sum{\Bbb P}[e_i]{\Bbb P}^\mu[\phi(\bar x, m_i)].$$
As this is only defined for formulae over the parameter set $(M\otimes{\cal A})_0$,  it can be  extended by
continuity to the whole structure. (Cf. Lemma~3.8 in  \cite{CGH-generic} for an exact statement of this point.) The type $\mu\otimes{\cal A}$ is called the {\em natural extension of $\mu$.} In \cite[Fact 3.10]{CGH-generic} and \cite[Def. 1.3]{Gannon transfer}, assuming that $\mu$ is definable, this type is expanded on the monster model $\Bbb U$ and is denoted by $r_\mu$. (Notice that in the present paper, we defined $\mu\otimes{\cal A}$ as a type on models of the form $M\otimes{\cal A}$ but not $\Bbb U$. This is an important point and is discussed in Remark \ref{finite satisfiable} below.)

\begin{Fact} \label{Ben-fact}
	Let $\cal U$ be the monster model of $T$, $\mu$ a global measure, and $\cal A$ an atomless measure algebra. Suppose that $M$ is a separable model of $T$ and  ${\cal A}_0\preceq{\cal A}$ is separable.
\newline
(i) $M\otimes{\cal A}_0\preceq{\cal U}\otimes {\cal A}$.
\newline
(ii) $\mu$ is definable over $M$ if and only if the type $\mu\otimes {\cal A}$ is definable over $M$.
\newline
(iii) $\mu$ is finitely satisfied in $M$  if and only if the type $\mu\otimes {\cal A}$ is finitely
satisfied in $M\otimes{\cal A}_0$.
\end{Fact}
\begin{proof} 
This is proved in Proposition 1.1 of \cite{Ben-transfer}.
For an accurate  proof of (ii), see Fact~3.10 and Remark~3.11 of \cite{CGH-generic}.
\end{proof}

\begin{Remark}\label{finite satisfiable}
 There is an important point regarding Fact~\ref{Ben-fact}(iii), and that is: In general, the model $\cal U\otimes\cal A$ is a small model of $T^R$, because this model is often not sufficiently saturated. (This means that ${\cal U}\otimes{\cal A}\prec{\Bbb U}$ but ${\cal U}\otimes{\cal A}\neq{\Bbb U}$. See Remark~\ref{warning}(ii), for an example.)
  Therefore,  $\mu\otimes{\cal A}$ is a  type on ${\cal U}\otimes\cal A$ but not a type on the monster model $\Bbb U$ of $T^R$.  Thus, if $\mu$ is finitely satisfiable then so is $\mu\otimes\cal A$, but the type $\mu\upharpoonright^{\Bbb U}$ on $\Bbb U$ (as defined in Subsection 3.3 below) may not be finitely satisfiable.
 In \cite[Def. 1.4]{Gannon transfer}, the type $\mu\upharpoonright^{\Bbb U}$ is denoted by $s_\mu$, and in Proposition~2.2 in there, it is shown that for a definable measure $\mu$, $s_\mu=r_\mu$.
 
 In Proposition 3.26 of \cite{CGH-generic}, 
 the authors introduced a type $q$ that is definable and finitely satisfiable, but $r_q$ is not finitely satisfiable. This does not contradict Fact \ref{Ben-fact}(iii), because $r_q$ is a  definable global type (i.e. over $\Bbb U$), and its restriction to ${\cal U}\otimes\cal A$ becomes $q\otimes {\cal A}$.
 Notice that, by Fact  \ref{Ben-fact}(iii), $q\otimes\cal A$ is finitely satisfiable and    has a finitely satisfiable extension to a global type (over $\Bbb U$), which we denote by $R_q$. Clearly $R_q\neq r_q$ and $R_q$ is not definable.
 To summarise, $R_q$ is a finitely satisfiable extension of $q\otimes\cal A$ and $r_q$ is a definable extension of $q\otimes\cal A$.\footnote{Thanks to James Hanson for clarifying this point for us.}

\end{Remark}

The observation below is nothing more than translating random-types to Keisler measures, which, considering the previous discussions, is enlightening and can lead to further results. (Cf. Remarks~\ref{weak reverse} and \ref{warning}.)
\begin{Proposition}  \label{main}
	Let $T$ be a (countable) classical theory, $M$ a small model of $T$, and $\mu(x)$ a global $M$-invariant measure.
	 Let ${\cal A}$ be an atomless measure algebra such that $[0,1]\preceq{\cal A}$. Then (i)~$\Rightarrow$~(ii):
	\newline
	(i) $\mu$ is definable over $M$ and  there is a sequence $({\bf a}_n)\in M\otimes [0,1]$ such that
for any $L$-formula $\phi(x,y)$ and every ${\bf b}\in{\cal U}\otimes\cal A$ the sequence 
 ${\Bbb P}[(\phi({\bf a}_n,{\bf b})]:n<\omega)$ Baire-1/2-converges to ${\Bbb P}[\phi(x,{\bf b})]^{\mu\otimes{\cal A}}$.
 \newline
	(ii)  $\mu$ is $R$-generically stable ($fam$) over $M$ (as in Definition~\ref{generic-measure}).
\end{Proposition}
\begin{proof}
	We can assume that $M$ is separable. Let $M_0\subseteq M$ be dense and countable, and ${\cal A}_0\subseteq[0,1]$ be dense and countable. By Fact~\ref{Ben-fact}(ii), $\mu$ is  definable over $M$ if and only if  $\mu\otimes{\cal A}$ is    definable over $M\otimes [0,1]$. The point of referencing to Fact~\ref{Ben-fact} is to be able to assume definability when proving  (1) implies (2) below:

	\medskip\noindent 
	(1): There is a sequence $({\bf a}_n)\in M_0\otimes {\cal A}_0$ such that for any $L$-formula $\phi(x,y)$
and every ${\bf b}\in{\cal U}\otimes\cal A$ the sequence  $({\Bbb P}[\phi({\bf a}_n,{\bf b})]:n<\omega)$ Baire-1/2-converges to ${\Bbb P}[\phi(x,{\bf b})]^{\mu\otimes\cal A}$.
\newline
(2): There is a sequence $(\mu_n=\frac{1}{k_n}\sum_{i=1}^{k_n}p_{n,i} :n<\omega)$, where $a_{n,i}\models p_{n,i}$ and $a_{n,i}\in M_0$ (for all $n$ and $i$), such that $(\mu_n:n<\omega)$  randomly converges to $\mu$.

\medskip\noindent 
(1) $\Rightarrow$ (2): Assume that (1) holds. Then, by Definition \ref{Baire-1/2-convergence 2}, for any $\phi(x,y)\in L$, and for each $r<s$, there is a natural number $N=N_{r,s}^{{\Bbb P}[\phi]}$ and a set $E\subseteq\{1,\ldots,N\}$ such that for each $i_1<\cdots<i_N<\omega$, and any parameter ${\bf b}\in{\cal U}\otimes\cal A$, the following does not hold:
		$$\bigwedge_{j\in E}{\Bbb P}[\phi({\bf a}_{i_j},{\bf b})]\leq r ~\wedge~\bigwedge_{j\in N\setminus E}{\Bbb P}[\phi({\bf a}_{i_j},{\bf b})]\geq s.  \ \ \ \Diamond$$ 
Notice that every ${\bf a}_n$ is of the form in $\sum_{i\leq k_n} a_{n,i}\otimes e_{n,i}$ where $a_{n,1},\ldots,a_{n,k_n}\in M_0$ and $e_{n,1},\ldots  e_{n,k_n}$ is a measurable partition of ${\cal A}_0$ (such that the measure of $e_{n,i}$ is $\frac{1}{k_n}$ for all $i\leq k_n$). Set $\mu_n:=\sum_{i\leq k_n} a_{n,i}$. That is, $\mu_n(\phi(x,b))=\frac{1}{k_n}\sum_{i\leq k_n} \phi(a_{n,i},b)$ for all $b\in\cal U$.  Let ${\bf b}\in({\cal U}\otimes{\cal A})_0$ be arbitrary. By  the construction, $\bf b$ is of the form in  $\sum_{j\leq m} b_j\otimes e_j$ where $b_1,\ldots,b_m\in\cal U$ and $e_1,\ldots  e_m$ is a measurable partition of $\cal A$. With  the following computation we check that the condition (2) holds:
\begin{align*}
{\Bbb P}[\phi({\bf a}_n,{\bf b})]  &  = {\Bbb P}[\phi({\bf a}_n,\sum_{j\leq m} b_j  e_j)]= \sum_j{\Bbb P}[e_j]\cdot{\Bbb P}[\phi({\bf a}_n, b_j)] \\
& =  \sum_j{\Bbb P}[e_j]\cdot{\Bbb P}[\phi(\sum_{i\leq k_n} a_{n,i}e_{n,i}, b_j)] \\
& =  \sum_j{\Bbb P}[e_j]\cdotp\Big(\sum_i{\Bbb P}[e_{n,i}]\cdotp\phi(a_{n,i},b_j) \Big)  \\
& =  \sum_j\frac{1}{m}\cdotp\Big(\sum_i\frac{1}{k_n}\cdotp\phi(a_{n,i},b_j) \Big) \ \ \ (*)  
\end{align*}
(Recall from Convention~\ref{simple writing} that in the above we wrote $a_{n,i}$ and $b_j$ instead of $a_{n,i}\otimes1$ and $ b_j\otimes1$.)
 
Now, define the Keisler measure $\mu_n:=\frac{1}{k_n}\sum_{i=1}^{k_n}p_{n,i}$, where $a_{n,i}\models p_{n,i}$  ($i\leq k_n$). By ($*$), we have 
\begin{align*}
{\Bbb P}[\phi({\bf a}_n,{\bf b})]  &  = \sum_j\frac{1}{m}\cdotp\Big(\sum_i\frac{1}{k_n}\cdotp\phi(a_i,b_j) \Big) \\
& = \mu_n^{\sum r_j}(\phi;(b_j)_j^{m}) \ \ \ \text{ where } r_j=\frac{1}{m} \text{ for any } j\leq m
\end{align*} 
(Cf. Subsection~3.1 for  the notation $\mu^{\sum}$.)

Therefore, by $\Diamond$, for each $i_1<\cdots<i_N<\omega$, and any parameter $b_1,\ldots,b_m\in M_0$, the following does not hold:
$$\bigwedge_{j\in E} \mu_{i_j}^{\sum r_t}(\phi;(b_t)_t^{m})\leq r ~\wedge~\bigwedge_{j\in N\setminus E}\mu_{i_j}^{\sum r_t}(\phi;(b_t)_t^{m})\geq s.$$ 
\noindent
This means that  the sequence $(\mu_n=\frac{1}{k_n}\sum_{i=1}^{k_n}p_{n,i} :n<\omega)$   randomly converges. It easy to see that its limit is $\mu$ and was left to the reader. 
As $({\cal U}\otimes{\cal A})_0$ is metrically dense in ${\cal U}\otimes{\cal A}$, the proof is completed.
\end{proof}
\medskip
It seems that we showed something stronger than (i) $\Rightarrow$ (ii), since in the argument the ${\bf a}_n$'s are in $M_0\otimes{\cal A}_0$. Indeed, notice that for each $n$, we can replace ${\bf c}_n\in M\otimes [0,1]$ with ${\bf a}_n\in M_0\otimes {\cal A}_0$ such  that $d({\bf c}_n,{\bf a}_n)<\frac{1}{n}$. Now, $({\bf c}_n)$ Baire-1/2 converges   if and only if   $({\bf a}_n)$ Baire-1/2 converges. Therefore,  (i) $\Rightarrow$ (ii) is exactly equivalent to what we proved.
\medskip

\begin{Remark} \label{weak reverse}
	In general, the direction (ii)~$\Rightarrow$~(i) in Proposition~\ref{main} does not hold. 
(Cf. Remark~\ref{warning} below for some explanation, and Example~\ref{example 2}.) However, it is still possible to prove something weaker. Suppose that (ii) holds. That is, there is a sequence $(\mu_n=\frac{1}{k_n}\sum_{i=1}^{k_n}p_{n,i} :n<\omega)$, where $a_{n,i}\models p_{n,i}$ and $a_{n,i}\in M_0$ (for all $n$ and $i$), such that $(\mu_n:n<\omega)$  randomly converges to $\mu$. Let ${\cal N}={\cal U}\otimes({\cal A}\times[0,1))$ be the model presented in  \cite[Fact~2.3]{Gannon transfer} which is denoted by ${\cal U}^{\Omega\times[0,1)}$ in there. Define ${\bf a}_n\in\cal N$ as follows: ${\bf a}_n(t,s)=a_{n,i}$ if $s\in[\frac{i-1}{k_n},\frac{i}{k_n})$. Then, it is easy to see that for any $L$-formula $\phi(x,y)$ and every ${\bf b}\in{\cal U}\otimes\cal A$ the sequence 
	${\Bbb P}[(\phi({\bf a}_n,{\bf b})]:n<\omega)$ Baire-1/2-converges to ${\Bbb P}[\phi(x,{\bf b})]^{\mu\otimes{\cal A}}$. Obviously, this result is weaker than the condition (i), because the  ${\bf a}_n$'s are in $\cal N$ and may not be in ${\cal U}\otimes\cal A$.

On the other hand, similar to the argument of Proposition \ref{main}, one can show that if the corresponding random-type $\mu\upharpoonright^{\Bbb U}$ (as defined in the next subsection) is generically stable over $M\otimes\cal A$, then $\mu$ is $fam$ over $M$. (Cf. Theorem~\ref{fim=generic} below.)
\end{Remark}

\begin{Remark} \label{warning}
(i)  As mentioned above, in general, the direction (ii)~$\Rightarrow$~(i) of Proposition~\ref{main} does not hold, and so one can  not  expect that a combination of Proposition~\ref{main} and Theorem~\ref{generic type continuous} implies that a measure $\mu$ (in classical logic)  is $R$-generically stable if and only if ``its randomization" is generically stable (in continuous logic). The reasons are that: (1) In general,  ${\cal U}\otimes\cal A$ is not 
the monster model $\Bbb U$ of $T^R$. 
(See the following example by James Hanson on (non-)saturation of models of the form $M\otimes\mathcal{A}$.) 
  And (2) the randomization of average Keisler measures (as defined in the next subsection) may not be realized in the monster model $\Bbb U$. This is related to the fact that the direction (ii)~$\Rightarrow$~(i) in Proposition~\ref{main} does not hold in general. (See Remark~\ref{weak reverse} above and also Remark~3.24 in \cite{CGH-generic}.) 
	\newline
	(ii)  James Hanson pointed out to us that   randomizations of the form $M\otimes\mathcal{A}$ are not sufficiently saturated in general.\footnote{Although, probably simultaneously (and of course independently) we and Hanson realized that ${\cal U}\otimes\mathcal{A}$ could not be the monster model, it was Hanson who made this clear.}  For an easy example, let  $T$ be $DLO$ with constants added for $\Bbb Q$. Consider the Lebesgue measure on $[0,1]$ and let $p$ be the corresponding type in $T^R$. No model of $T^R$ of the form $M\otimes\mathcal{A}$ can realize $p$, because for any element $\bf a$ of such a model, $tp(\bf a)$ corresponds to a measure in $T$ with atoms. (And this is in the completion, not just the pre-model. Any type corresponding to a measure with finite support has distance $1$ from $p$ in $T^R$'s type space.) He also suggested a characterization of a theory $T$ such that  models of $T^R$ of the form $M\otimes\mathcal{A}$ can be arbitrarily saturated: a proper subclasses of stable theories.
\end{Remark}


\subsection{Global random-types and $NIP$}
In this subsection we first  review the notion of corresponding random-type $\mu\upharpoonright^{\Bbb U}$ of  a global measure $\mu$ (in classical logic),\footnote{This type is called {\em extension by definition} in \cite{Ben-transfer}.} and then study this notion and its connection with generic stability of measures in general/$NIP$ theories.
The content of this subsection is mostly descriptive, and much of it can be found in certain  articles. However, it is useful in that it clarifies certain concepts and highlights the relationships between them.


\begin{Definition} Let $T$ be a classical theory, $M\models T$ and $A\subseteq M$. 
Let $\mu(x)$ be a measure over $M$.
In the following, we assume at least one of the two cases:  (1) $\mu(x)$ is  {\em definable   over $A$}, 
 or (2) $T$ is $NIP$ and $\mu(x)$ is  {\em finitely satisfiable in $A$}. 

  Therefore, for any formula $\phi(x,y)$, there is a Borel  function $f_\mu^\phi:S_y(A)\to[0,1]$ such that  ${\Bbb P}^\mu[\phi(x,b)]=f_\mu^\phi  (q)$ where $q=tp(b/A)$.\footnote{Note that in case (1), the function $f_\mu^\phi$ is continuous, and in case (2), based on a result in \cite{HP}, it is Borel measurable.}

If $\Bbb M$ is a model of $T^R$ containing $M$ (e.g. ${\Bbb M}\succeq M\otimes\mathcal{A}$ or $\Bbb M=\Bbb U$ the monster model of $T^R$), there is an extension  
$\mu\upharpoonright^{\Bbb M}$ of $\mu$ over $\Bbb M$ as follows:
$${\Bbb P}[\phi(x,{\bf b})]^{\mu\upharpoonright^{\Bbb M}}=\int f_\mu^\phi(q) \ d \ tp({\bf b}/A),$$
where  $tp({\bf b}/A)$ is a Borel probability measure on $S_y(A)$. In  this paper, the type $\mu\upharpoonright^{\Bbb M}$ is called the {\em corresponding random-type of $\mu$}. In  \cite[Def. 1.4]{Gannon transfer}, this type  is denoted by $s_\mu$, and in Proposition~2.2 in there, it is shown that for a definable measure $\mu$, we have $s_\mu=r_\mu$. 
\end{Definition}

\begin{Remark}\label{Borel definable}
(i) If $\mu$ is definable, by Proposition 2.2 in \cite{Gannon transfer} and Fact 3.10 in \cite{CGH-generic}, $s_\mu$ is well-defined and consistent.
If $\mu$ is finitely satisfiable and $T$ is $NIP$, by Proposition 2.4 in \cite{Gannon transfer}, $s_\mu$ is well-defined and consistent.
\newline
(ii) In addition, if $\mu$ is {\em finitely satisfiable in $A$}, then $\mu\upharpoonright^{\Bbb M}$ can be defined as follows:
$${\Bbb P}[\phi(x,{\bf b})]^{\mu\upharpoonright^{\Bbb M}}=\int f_\mu^\phi(q) \ d \ tp_{\phi^*}({\bf b}/A),$$
where  $f_\mu^\phi:q\mapsto {\Bbb P}^\mu[\phi(x,b)]$ with $q=tp_{\phi^*}(b/A)$, is the defining function of $\mu$ on $S_{\phi^*}(A)$, and $tp_{\phi^*}({\bf b}/A)$ is a Borel probability measure on $S_{\phi^*}(A)$. (In this case, we say that the definition of $\mu$ is factors via $S_{\phi^*}(A)$. That is, ${\Bbb P}^\mu[\phi(x,b)]={\Bbb P}^\mu[\phi(x,b')]$   if and only if   $tp_{\phi^*}(b/A)=tp_{\phi^*}(b'/A)$.)
\newline
(iii) Notice that if $\Bbb M=M\otimes \cal A$, then $\mu\upharpoonright^{\Bbb M}=\mu\otimes \cal A$ (as in the previous section). This follows from Proposition 2.2 of \cite{Gannon transfer}.
\newline
\end{Remark}


  


\begin{Fact}  \label{Ben-fact-2}
 (Assuming $NIP$) 	Let $M$ be a model, $\mu$  a global measure,   and  ${\cal A}\prec[0,1]$ an atomless probability algebra.   If $\mu$  is finitely satisfiable in $M$,  then $\mu\upharpoonright^{\Bbb U}$ is well-defined, consistent and approximately finitely satisfiable in $M\otimes{\cal A}$.
\end{Fact}
\begin{proof}
This follows from \cite[Prop. 2.4]{Gannon transfer}.
\footnote{This fact was first proved in Proposition 2.1(ii) of \cite{Ben-transfer}  using a different method.} 
\end{proof}

  In \cite{K-GC}, using a crucial result due to Bourgain, Fremlin, and Talagrand, we gave an   alternative argument of \cite[Thm. 5.10]{Gannon sequential}. (See  Theorem~A.1 in \cite{K-GC}.) In fact,   \cite[Thm.~A.1]{K-GC} is a  refinement of \cite[Thm. 5.10]{Gannon sequential}.  
The following is another   argument. 
\begin{Theorem}  \label{Gannon-alternative-2}
 (Assuming $NIP$) 	Let $T$ be a countable theory and $M$ a countable model of $T$. Then, every global measure $\mu(x)$ which is finitely satisfiable in $M$ is the limit of a sequence of average types realised in $M$, that is, there is a sequence $(\bar{a}_n)\in M^{<\omega}$ such that for every formula $\phi(x,y)$,  $$\lim_{n\to\infty} \text{Av}(\bar{a}_n)(\phi(x,b))=\mu(\phi(x,b)) \   \text{ for all } b\in\cal U.$$
\end{Theorem}
\begin{proof}
Let ${\cal A}\prec[0,1]$ be any atomless probability algebra. 
As $T$ has $NIP$,  the corresponding random type  $\mu\upharpoonright^{\Bbb U}$ is finitely satisfiable in $M\otimes {\cal A}$ by Fact~\ref{Ben-fact-2}. 
 Again, as $T$ is  $NIP$, by \cite[Thm.  5.3]{Ben-VC}, $T^R$ is $NIP$. Therefore, some/any Morley sequence of  $\mu\upharpoonright^{\Bbb U}$ (over  $M\otimes {\cal A}$) is convergent.
  By Theorem~\ref{Morley sequence}, there is a sequence in $M\otimes {\cal A}$ which is (Baire-1/2) convergent to  $\mu\upharpoonright^{\Bbb U}$. Now, by a translation similar to the argument of Proposition~\ref{main}, it is easy to see  that there is a sequence of averages measures of realized types in $M$ which converges to $\mu$. This is enough. (Furthermore, such a  sequence is randomly convergent.)
\end{proof}

\begin{Fact}  \label{main completion}  (Assuming $NIP$)
	Let $T$ be a (countable) classical theory, $M$ a small model  of $T$, and $\mu(x)$ a global  measure which is definable over $M$.
	Let ${\cal A}$ be an atomless measure algebra such that $[0,1]\preceq{\cal A}$.  Then (i)~$\Longrightarrow$~(ii).
	\newline
	(i) $\mu$ is $R$-generically stable over $M$.
	\newline
	(ii)  
$\mu\upharpoonright^{\Bbb U}$ is generically stable over $M\otimes [0,1]$.
\end{Fact}
\begin{proof} First note that as $\mu$ is $fam$ over $M$, $\mu\upharpoonright^{\Bbb U}$ is definable over, and finitely satisfiable in $M\otimes[0,1]$. (Cf. Fact \ref{Ben-fact-2}.) As $T^R$ is $NIP$, some/any Morley sequence of $\mu\upharpoonright^{\Bbb U}$ is convergent. By Theorem~\ref{Morley sequence}, there is a sequence $({\bf a}_n:n<\omega)$ in $M\otimes[0,1]$ such that $\lim tp({\bf a}_n/{\Bbb U})=\mu\upharpoonright^{\Bbb U}$. Furthermore, by Theorem~\ref{generic type continuous}, this measure is generically stable.
\end{proof}

\begin{Remark} Notice that,  
assuming $NIP$, Fact~\ref{main completion} provides a reverse to Proposition~\ref{main}. Indeed, by Theorem \ref{generic type continuous}, the condition (ii) of Fact~\ref{main completion} imples the condition (i) of Proposition~\ref{main}.
\end{Remark}

\medskip
The following is  a kind of complement to  Proposition~\ref{main} for $\mu\upharpoonright^{\Bbb U}$.
\begin{Theorem} \label{fim=generic}
	Let $T$ be a (countable)  theory, $M$ a model of $T$, and   $\mu(x)$  a global measure (in $T$) which is definable over $M$. Suppose that $\mu\upharpoonright^{\Bbb U}$ is generically stable over $M\otimes\mathcal{A}$. Then,  $\mu$ is $R$-generically stable ($fam$) over $M$.
\end{Theorem}
\begin{proof}
First notice that, as $\mu$ is definable, by \cite[Prop. 2.2]{Gannon transfer}, $\mu\upharpoonright^{\Bbb U}$ is well-defined and consistent. Furthermore, recall that $\mu\upharpoonright^{{\cal U}\otimes \cal A}=\mu\otimes\cal A$.
 By Theorem~\ref{generic type continuous},  there is a sequence $({\bf a}_n)\in M\otimes [0,1]$ such that the sequence $(tp({\bf a}_n/{\Bbb U}) :n<\omega)$ Baire-1/2-converges to $\mu\upharpoonright^{\Bbb U}$. In particular,  the sequence $(tp({\bf a}_n/{\cal U}\otimes{\cal A}):n<\omega)$ Baire-1/2-converges to $\mu\otimes{\cal A}$. Therefore, by  Proposition~\ref{main},  $\mu$ is $R$-generically stable over $M$.
\end{proof}
The following example shows that one can not expect a converse to Theorem~\ref{fim=generic}.
\begin{Example} \label{example 2}   Consider the unique complete type $p\supset\{\neg xEb:b\in{\cal U}\}$ in the Henson graph. Conant and Gannon \cite{CG} showed that this type is $fam$ but not $fim$. Using the methods of the present paper it is easy to see  that randomization of $fam$ measures are $fam$ again. (This has been proven in \cite[Thm 3.25]{CGH-generic} with a different  approach than here.)
Also, in Corollary~\ref{type-fim=fim}, we will show that randomization of generically stable {\bf types} are generically stable. (See \cite[Corollary 3.19]{CGH-generic} for a different proof.)
To summarize, the Dirac measure  $\delta_p$ is $fam$ and its corresponding random-type  $\delta_p\upharpoonright^{\Bbb U}$ is not generically stable. 
\end{Example}

\section{Continuous $VC$-theory and generically stable types}
In this section, using powerful tools provided in \cite{Ben-VC}, we
can refine some of the previous results.
We  give an argument of the fact that the generic stability for {\bf types} is preserved in randomization (cf. Corollary~\ref{type-fim=fim}).

\medskip
We first recall some notion and notation from \cite{Ben-VC}. 
Let $I$ be a set, and $Q\subseteq[0,1]^I$  a collection of functions from $I$ to $[0,1]$.  For any $r<s$, there is a collection $Q_{r,s}=\{q_{r,s}:q\in Q\}$ of fuzzy sets, where the $q_{r,s}$'s are defined as follows: $i\in q_{r,s}$ if $q(i)\leq r$ and  $i\notin q_{r,s}$ if $q(i)\geq s$; and it is not  known/important that $i$ belongs to $q_{r,s}$ or not if $r<g(i)<s$.
The $VC$-index for  $Q_{r,s}$, denoted by $VC(Q_{r,s})$, is defined similar to classical case. (We refer to \cite[page 316]{Ben-VC} for  a precise definition of fuzzy sets and the $VC$-indexes of $Q$.)
$Q$ is called a $VC$-class if for any $r<s$, $VC(Q_{r,s})<\infty$. 
If $Q$ is a $VC$-class, one can easily show that for every $\epsilon>0$ there is an  upper bound $d_\epsilon<\infty$ for the $VC$-indexes of classes $Q_{r,r+\epsilon}$ where $r\in[0,1-\epsilon]$.  
To summarize, the notion of a $VC$-class is such that we have the following important observations/connections:
\begin{Observation}   \label{observation}
	Let   $J$ be a set, and $\varphi:{\Bbb N}\times J\to[0,1]$ a function. Then, (i)~$\iff$~(ii)~$\Rightarrow$~(iii).
	\newline (i): $\varphi^J=\{\varphi(\cdot,j):{\Bbb N}\to[0,1] \ | \ j\in  J\}$ is a $VC$-class.
	\newline (ii): $\varphi_{\Bbb N}=\{\varphi(i,\cdot):J\to[0,1] \ | \ i\in {\Bbb N}\}$ is dependent. That is,   for every $\epsilon>0$ there is a natural number $N=N_\epsilon$ such that for every $r<s$ with $s-r=\epsilon$ the following does not hold  $$(*) \ \ \ \exists F\subseteq {\Bbb N}, |F|=N \ \forall E\subseteq F \  \exists j\in J \ \Big(\bigwedge_{i\in E}\varphi(i,j)\leq r \ \wedge \ \bigwedge_{i\in F\setminus E}\varphi(i,j)\geq s \Big).$$
(iii) There is a subsequence $(\varphi_n:n<\omega)$ of $\varphi_{\Bbb N}$ such that $(\varphi_n:n<\omega)$ Baire-1/2-converges. That is, for each $r<s$, there is a natural number $N=N_{r,s}$ and a set $E\subset\{1,\ldots,N\}$ such that  for  each $i_1<\cdots<i_{N}<\omega$,  the following does not hold
$$\exists j\in J \ \Big(\bigwedge_{k\in E}\varphi(i_k,j)\leq r~\wedge~\bigwedge_{k\in N\setminus E}\varphi(i_k,j)\geq s\Big).$$ 
\end{Observation}
\begin{proof}
The equivalence (i)~$\iff$~(ii) is folklore in classical logic (cf. \cite{van}). The argument is a  straightforward adaptation of the classical case.

(ii)~$\Rightarrow$~(iii) is similar to the argument of the direction (i)~$\Rightarrow$~(iii) of Proposition~2.14 in \cite{K-Baire}. Indeed, suppose for a contradiction that there is no Baire-1/2-convergent subsequence. Using Ramsey theorem, let $(\varphi_n:n<\omega)$ be a   $\varphi$-$N$-$A$-indiscernible sequenence as in Definition~3.1 of \cite{K-classification}. (Here $N$ is a natural number  and $\{r,s\}=A$.) For suitable   $N$  and $A$, the condition $(*)$ above holds for $(\varphi_n:n<\omega)$, because this sequenece is not  Baire-1/2-convergent.   As $N$ is arbitrary, this contradicts (ii).
\end{proof}
\begin{Observation} \label{observation 2}
With the above notation, every Baire-1/2-convergent sequence  $\{\varphi(n,\cdot):n\in{\Bbb N}\}$   is dependent (as in (ii) in Observation~\ref{observation}). 
\end{Observation}
\begin{proof}
Immediate.
\end{proof}
Compare to   Definitions~\ref{Baire-1/2-convergence 2},  \ref{Baire-1/2-convergence}, and Remark~\ref{Baire 1/2 classic}.


\medskip
Fix a probability space $(\Omega,{\frak B},\mu)$. Given sets $I$ and $J$, a family of $[0,1]$-valued functions on $I\times J$ is given as $\varphi:\Omega\to[0,1]^{I\times J}$ by $\omega\mapsto\varphi_\omega(\cdot,\cdot)$. 
For every $\omega$, $(\varphi_\omega)_I=\{(\varphi_\omega)_i:i\in I\}=\{\varphi_\omega(i,\cdot):i\in I\}$ and similarly  $(\varphi_\omega)^J=\{(\varphi_\omega)^j:j\in J\}$.
The family $\varphi=\{\varphi_\omega:\omega\in\Omega\}$ is called {\em uniformly  dependent} if for every $\epsilon>0$ there is   $d=d_{\varphi,\epsilon}$ such that $VC\big(((\varphi_\omega)^J)_{{r,r+\epsilon}}\big)\leq d$ for every $r\in[0,1-\epsilon]$ and $\omega\in\Omega$.

We say that  $\varphi=\{\varphi_\omega:\omega\in\Omega\}$ is a measurable family  if for any $(i,j)\in I\times J$, the function $\omega\mapsto\varphi_\omega(i,j)$ is measurable. Then, we define function ${\Bbb P}[\varphi]:I\times J\to[0,1]$ by  ${\Bbb P}[\varphi](i,j):={\Bbb P}[\varphi(i,j)]$.

We say that ${\Bbb P}[\varphi]:I\times J\to[0,1]$ is dependent, if for every $r<s$, the collection  $({\Bbb P}[\varphi]^J)_{r,s}$ is a $VC$-class. (Cf. \cite{Ben-VC}, page 316 and Proposition~2.15.)

\medskip
\noindent
{\bf Explanation.} We explain how ${\Bbb P}[\phi]$ will be interpreted in the next theorem (\ref{Ben local}). Let $\varphi(x,y)$ be an $L$-formula and $A,B$ subsets of a model ${\bf M}\models T^R$.  Recall from the construction of randomization that the universe of every model $\bf M$ of $T^R$ is a set of measurable  maps from an atomless probability algebra $\Omega$ to a model $M$ of $T$. This means that for each ${\bf a}\in\bf M$ and every $\omega\in\Omega$, we have ${\bf a}(\omega)\in  M$.

 Let us enumerate $A=\{{\bf a}_i:i\in I\}$ and $B=\{{\bf b}_j:j\in J\}$. Set $\psi(x,y)={\Bbb P}[\varphi(x,y)]$, i.e it is the corresponding formula of $\varphi$ in randomization. Define $\chi_\varphi:I\times J\to [0,1]$ via $\chi_\varphi(i,j):=\psi({\bf a}_i,{\bf b}_j)$. Therefore,
\begin{align*}
{\Bbb P}[\chi_\varphi](i,j) & ={\Bbb P}[\chi_\varphi(i,j)]=\psi({\bf a}_i,{\bf b}_j) \\ &
={\Bbb P}[\varphi({\bf a}_i,{\bf b}_j)]={\Bbb  P}\big(\{\omega\in\Omega:\models \varphi({\bf a}_i(\omega),{\bf b}_j(\omega))\}\big).
\end{align*}
In particular, recall that $M$ is a subset of $\bf M$ via $a\mapsto a\otimes 1$, where $a\otimes 1$ is the constant map $\omega\mapsto a$. In this case, as mentioned in Convention~\ref{simple writing}, to simplify the notation, we continue to use $a$ instead of $a\otimes 1$.

\medskip
The key result is the following.
\begin{Fact}[\cite{Ben-VC}, Corollary~4.2]  \label{Ben key}  
	If  $\varphi=\{\varphi_\omega:\omega\in\Omega\}$ is a measurable family of uniformly dependent functions, then 
${\Bbb P}[\varphi]:I\times J\to[0,1]$ is dependent.
\end{Fact}

The following localizes Theorem~5.3 of \cite{Ben-VC} in two way: for a formula and   a sequence. (We strongly suggest that the proof of   \cite[Thm. 5.3]{Ben-VC} should be  read before reading the rest of this section.)
\begin{Theorem} \label{Ben local}
Let $M$ be a  model of $T$,   $\psi(\bar x,\bar y)$ a formula with $|\bar x|=n$, $|\bar y|=m$, and $(a_i)$ a sequence in $M$ of $n$-tuples. If the sequence  $(\psi(a_i,\bar y):i<\omega)$ $DBSC$-converges, then the sequence $({\Bbb P}[\psi(a_i,\bar y)]:i<\omega)$ Baire-1/2-converges. (Here ${\Bbb P}[\psi(\bar x,\bar y)]$ is the corresponding formula in $T^R$.)
\end{Theorem}
\begin{proof} The proof is an adaptation of the argument of \cite[Thm. 5.3]{Ben-VC}. We write $\varphi(\bar x,\bar y)={\Bbb P}[\psi(\bar x,\bar y)]$.
 We will show that $\varphi^{\Bbb U}$ is dependent on
$\{a_i:i<\omega\}\times {\Bbb U}^m$. 
Let us enumerate
 ${\Bbb U}^m = \{\bar {\bf b}_j\colon j \in J\}$, and set $A=\{a_i:i<\omega\}$.
Let ${\bf p} = tp(A , {\Bbb U}^m/\emptyset)$.
We may write it as ${\bf p}( \bar x_i, \bar y_j )_{i\in \omega,j \in J} \in S_{\omega \cup J}(T^R)$,
and identify it with a probability measure
$\mu$ on $\Omega = S_{(\omega\times n) \cup (J\times m)}(T)$ such that for every formula $\rho(\bar z)$
of the theory $T$, $\bar z \subseteq \{\bar x_i,\bar y_j\}_{i\in \omega,j\in J}$:
\begin{gather*}
{\Bbb P}[\rho(\bar z)]^{\bf p} =\mu\big(\{q\in\Omega: \rho(\bar z)\in q\}\big)= \int_\Omega \rho(\bar z)^q\, d\mu(q).
\end{gather*}
We can replace $\Omega$ by $\Omega':=\{q\in\Omega:  (c_i)\models tp(A/\emptyset) \text{ for any } (c_i)_{i<\omega}\cup(d_j)_{j\in J}\models q\}$, and assume that $\mu$ is concentrated on $\Omega'$.
In fact, we can assume that ${\Bbb P}[\rho(\bar z)]^{\bf p} =  \int_{S_{(J\times m)}(T)} \rho(\bar a_i, \bar y_j)^{q_j}\, d\mu(q_j)=\int_{\Omega'} \rho(\bar x_i, \bar y_j)^q\, d\mu(q).$

For $i \in \Bbb N$, $j \in J$ and $q \in \Omega'$ define:
$\chi_q(i,j) = \psi(\bar x_i,\bar y_j)^q$.
Then $\chi = \{\chi_q\colon q \in \Omega'\}$ is a measurable family of
$\{0,1\}$-valued functions on ${\Bbb N} \times J$
and $\varphi(\bar a_i,\bar {\bf b}_j) = \varphi(\bar x_i,\bar y_j)^{\bf p} ={\Bbb P}[\psi(\bar x_i, \bar y_j)]= \Bbb P[\chi](i,j)$
where expectation is
with respect to $\mu$. (Notice that, as $\psi(\bar x,\bar y)$ is a formula, for  $(i,j)\in {\Bbb N}\times J$, the functions  $q\mapsto \chi_q(i,j)$ are measurable.)
Then, by Observations ~\ref{observation} and \ref{observation 2} and the assumption of $DBSC$-convergence,   the family $\{\chi_q\colon q \in \Omega'\}$ is uniformly
dependent. 
  That is, for any $\epsilon>0$, there is $d=d_{\varphi,\epsilon}$ such that $VC((\chi_q)^J)_{[r,r+\epsilon]}\leq d$ for all $r\in[0,1-\epsilon]$ and $q\in\Omega'$.
Therefore, by Fact~\ref{Ben key}, $\Bbb P[\chi] \colon {\Bbb N}\times J \to [0,1]$ is dependent. By Observation~\ref{observation}, this means that $\{\varphi(a_i,y):{\Bbb U}^m\to[0,1]~|i\in {\Bbb N}\}$ is dependent.
Therefore, by Observation~\ref{observation}(iii), there is a subsequence $({\Bbb P}[\psi(a_{i_j},\bar y)]:j<\omega)$ which is Baire-1/2-convergent. 
\end{proof}

Since $fim$ and generic stability are equivalent for types, the direction (i) to (ii) of the following result is a case of \cite[Corollary 3.19]{CGH}. However,  our method is completely different and enlightening in some ways.
\begin{Corollary} \label{type-fim=fim}
	Let   $p(x)$ be a global type (in $T$) which is definable. Then the following are equivalent:
	\newline
	(i) $p$ is   generically stable.
	\newline
	(ii) There is a sequence $(a_i)\in\cal U$ such that  $(tp(a_i/{\Bbb U}):i<\omega)$   Baire-1/2-converges to $p\upharpoonright^{\Bbb U}$. (Therefore, $p\upharpoonright^{\Bbb U}$  is generically stable.)
\end{Corollary}
\begin{proof}
(i)~$\Longrightarrow$~(ii) follows from Theorem~\ref{Ben local}. Indeed, by Theorem~4.4 of \cite{K-Morley}, there is a sequence $(a_i)\in\cal U$ such that $(tp(a_i/{\cal U}):i<\omega)$ $DBSC$-converges to $p$. By Theorem~\ref{Ben local} above, $(tp(a_i/{\Bbb U}):i<\omega)$ Baire-1/2-converges to $p\upharpoonright^{\Bbb U}$.
Indeed, notice that,  by Theorem~\ref{Ben local}, for every $L$-formula $\phi(x,y)$, the sequence $({\Bbb P}[\phi(a_i,y)]:i<\omega)$ Baire-1/2-converges to ${\Bbb P}[\phi(x,y)]^{p\upharpoonright^{\Bbb U}}$.
By the elimination of quantifiers in $T^R$, the same holds for every formula in $L^R$ with variable $x$. (More precisely, it is easy to see that if $\varphi_1,\varphi_2$ are formulas of the form in ${\Bbb P}[\phi_1]$ and ${\Bbb P}[\phi_2]$, then for every continuous  combination $\psi(x,y)$ of them, the sequence $(\psi(a_i,y):i<\omega)$  Baire-1/2 convergence to $\psi(x,y)^{p\upharpoonright^{\Bbb U}}$.
(For this, one can use  the argument Corollary~3.5 in \cite{Ben-VC}.)
This means that we proved the claim for every atomic formula. 
Furthermore, it is easy to see that the same holds for uniform limit of formulas.\footnote{Recall that in continuous logic, the concept of a formula has been extended, and the uniform limit of formulas is also considered a formula.}
Now, by the elimination of quantifiers, the claim is proved for every formula.)


(ii)~$\Longrightarrow$~(i): As $(tp(a_i/{\Bbb U}):i<\omega)$   Baire-1/2 converges to $p\upharpoonright^{\Bbb U}$, the sequence  $(tp(a_i/{\cal U}):i<\omega)$   Baire-1/2 converges to $p$. (This follows from an argument similar to Proposition~\ref{main}, even simpler.) Notice that, for types,   $DBSC$-convergence and Baire-1/2 convergence coincide. So, by Theorem~4.4 of \cite{K-Morley}, $p$ is generically stable.
\end{proof}

\begin{Remark} (i): In the previous version of the present paper we asked  whether the generic stability of $p\upharpoonright^{\Bbb U}$ implies the condition (ii) in Corollary~\ref{type-fim=fim}; equivalently, if $p\upharpoonright^{\Bbb U}$ is generically stable, then  is $p$ generically stable? In Proposition 3.20 of \cite{CGH-generic}, this question was answered positively. Therefore, a definable type in classical logic is generically stable if and only if its corresponding random-type is generically stable.
\newline
(ii) A question that naturally arises is why not use the method used in this section for measures  instead of just types. The answer is that the randomization of average measures is not necessarily realized in $\Bbb U$, something that does not happen for types, i.e. if $a$ realizes $p$, then $a\otimes 1$ realizes $p\upharpoonright^{\Bbb U}$.
\newline
(iii) The final point is that the result of this section also holds for continuous theories, meaning we could have assumed that $T$ is a countable separable theory.
\end{Remark}


\bigskip\noindent
{\bf Acknowledgements.}  I want to thank James Hanson   for his helpful comments (especially because of the point on non-saturation of models of the form $M\otimes\mathcal{A}$).

I would like to thank the Institute for Basic Sciences (IPM), Tehran, Iran. Research partially supported by IPM grant~1401030117.


\begin{thebibliography}{99} \label{ref}
	
	\bibitem[Ben09]{Ben-VC} I. Ben Yaacov. Continuous and Random Vapnik-Chervonenkis Classes. Israel Journal of Mathematics 173 (2009), 309-333
	
	
	
    \bibitem[Ben09a]{Ben-transfer}  I. Ben-Yaacov,    Transfer of properties between measures and random types, unpublished  research note, 2009. http://math.univ-lyon1.fr/~begnac/articles/MsrPrps.pdf
	
	\bibitem[Ben13]{Ben2} I. Ben-Yaacov, On theories of random variables, Israel J. Math. 194 (2013), no. 2, 957-1012
	

	
	
	\bibitem[BBHU08]{BBHU} I. Ben-Yaacov, A. Berenstein, C. W. Henson, A. Usvyatsov, \emph{Model theory for metric structures}, Model theory with Applications to Algebra and Analysis, vol. 2  (Z. Chatzidakis, D. Macpherson, A. Pillay, and A. Wilkie, eds.),  London Math Society Lecture Note Series, vol. 350, Cambridge University Press, 2008.
	
	\bibitem[BK09]{Ben-Keisler} I. Ben Yaacov, and H.J. Keisler,	Randomizations of models as metric structures. Confluentes Mathematici, 1:197–223 (2009).
	
	
	
	
	
	\bibitem[CG20]{CG} G. Conant, K. Gannon, Remarks on generic stability in independent theories, Ann. Pure	Appl. Logic 171 (2020), no. 2, 102736, 20. MR 4033642
	
	
	\bibitem[CGH23]{CGH-generic} G. Conant, K. Gannon, J. Hanson,  Generic stability, randomizations, and NIP formulas, arXive preprint, https://arxiv.org/abs/2308.01801,  (2023).
	
	\bibitem[CGH23a]{CGH} G. Conant, K. Gannon, J. Hanson,  Keisler measures in the wild, Model Theory 2 (1), 1-67 (2023).
	
	
	
	
	
	
	
	
	
	
	
	
	
	\bibitem[G21]{Gannon sequential} K. Gannon, Sequential approximations for types and Keisler measures, preprint:  2021 https://arxiv.org/abs/2103.09946v2
	
	
	\bibitem[G24]{Gannon transfer} K. Gannon, A note on transfer maps and the Morley product in NIP theories, arXiv  preprint, https://arxiv.org/abs/2402.04202v2, 2024.
	
	\bibitem[G52]{Grothendieck}	A. Grothendieck, Criteres de Compacite dans les Espaces Fonctionnels Generaux, Am. J. Math, 74 (1952), 168-186.
	
	
	
	
	
	\bibitem[HP11]{HP} E. Hrushovski,  A. Pillay, On $NIP$ and invariant measures, Journal of the European Mathematical Society, 13 (2011), 1005-1061.
	
	\bibitem[HPS13]{HPS} E. Hrushovski, A.  Pillay, and P. Simon. Generically stable and smooth measures in NIP	theories. Transactions of the American Mathematical Society 365.5 (2013): 2341-2366.
	
	
	
	
	
	
	
	
     \bibitem[Kei99]{Keisler}	H. Jerome Keisler. Randomizing a Model. Advances in Math 143 (1999), 124-158.
     
	
	
	
	\bibitem[Kha20]{K3} K. Khanaki,  Stability, the NIP, and the NSOP; Model Theoretic Properties of Formulas via Topological Properties of Function Spaces,  Math. Log. Quart. 66, No. 2, 136-149  (2020) / DOI 10.1002/malq.201500059
	
	 \bibitem[Kha20a]{K-classification} K. Khanaki,  On classification of continuous first order theories,  submitted, 2020.
	
	
	\bibitem[Kha21]{K-Dependent} K. Khanaki,  Dependent measures in independent theories,   submitted, 	https://arxiv.org/abs/2109.11973, 2021.
	
	
	\bibitem[Kha22]{K-Baire} K. Khanaki,  Dividing lines in unstable theories and subclasses of Baire 1 functions,  Archive for Mathematical Logic (2022), https://doi.org/10.1007/s00153-022-00816-8
	
	
	\bibitem[Kha23]{K-Morley} K. Khanaki,  Remarks on convergence of Morley sequences,   Journal of Symbolic Logic (2023),  https://doi.org/10.1017/jsl.2023.18
	
	
	\bibitem[Kha24]{K-GC} K. Khanaki,  Glivenko-Cantelli classes and $NIP$ formulas,   Archive for Mathematical Logic, (2024)  https://doi.org/10.1007/s00153-024-00932-7
	
		
		
	
	
	
	
	

	
	\bibitem[KP18]{KP} K. Khanaki, A. Pillay, Remarks on $NIP$ in a model, Math. Log. Quart. 64, No. 6, 429-434 (2018) / DOI 10.1002/malq.201700070
	
	
	
	
	\bibitem[PT11]{Pillay-Tanovic} A. Pillay and P.  Tanović, Generic stability, regularity, and quasiminimality, Models, logics, and higher-dimensional categories, CRM Proc. Lecture Notes, vol. 53, Amer. Math. Soc., Providence, RI, 2011, pp. 189–211. MR 2867971
	
	
\bibitem[Poi00]{P} B. Poizat, A Course in Model Theory, Springer 2000.
	
	
	
	
	
	\bibitem[Ros74]{Ros} H. P. Rosenthal,   A characterization of Banach spaces containing $l^1$, Proc. Nat. Acad. Sci. U.S.A. 71 (1974), 2411-2413.
	
	
	
	\bibitem[S15]{Simon} P. Simon. A guide to NIP theories. Cambridge University Press, 2015.
	
	\bibitem[S15a]{S-invariant} P. Simon, Invariant types in NIP theories, Journal of Mathematical Logic, (2015).
	
	
	
	
	
	
	\bibitem[van98]{van} Lou van den Dries, Tame topology and o-minimal structures, London Mathematical Society
	Lecture Note Series, vol. 248, Cambridge University Press, Cambridge, 1998
\end{thebibliography}
\end{document}